\documentclass[11pt,leqno]{jdg-p}

\makeatletter
  \renewcommand{\theequation}{%
         \thesection.\arabic{equation}}
  \@addtoreset{equation}{section}
\makeatother

\usepackage{bm}
\usepackage{amsmath,amsthm,amssymb,amstext,amsbsy,amsopn,amscd,amsxtra}

\newtheorem{thm}{Theorem}[section]
\newtheorem{lem}{Lemma}[section]
\newtheorem{prop}{Proposition}[section]

\newtheorem{rem}{Remark}[section]

\newtheorem{asum}{Assumption}[section]
\newtheorem{dfn}{Definition}[section]

\newcommand{\pd}{\partial}
\newcommand{\gm}{\gamma}
\newcommand{\Gm}{\Gamma}
\newcommand{\bn}{\bm{\nu}}

\newcommand{\lm}{\lambda}

\newcommand{\R}{\mathbb{R}}

\newcommand{\N}{\mathbb{N}}
\newcommand{\mx}[2]{\max_{#1}{#2}}
\newcommand{\mn}[2]{\min_{#1}{#2}}
\newcommand{\su}[2]{\sup_{#1}{#2}}
\newcommand{\nf}[2]{\inf_{#1}{#2}}

\newcommand{\av}[1]{\left| #1 \right|}

\newcommand{\Ln}[2]{\left\| #1 \right\|_{L^{#2}}}

\newcommand{\Lp}[2]{\left\| #1 \right\|_{#2}}
\newcommand{\Lg}[1]{\left\| #1 \right\|_{L^2_\gamma}}

\newcommand{\vk}{\kappa}
\newcommand{\vp}{\varphi}

\newcommand{\va}{\alpha}
\newcommand{\vb}{\beta}
\newcommand{\vd}{\delta}

\newcommand{\wSigma}{\widetilde\Sigma}
\newcommand{\mL}{\mathcal{L}}
\newcommand{\mH}{\mathcal{H}}
\newcommand{\mkL}{\mathfrak{L}}
\newcommand{\mR}{\mathcal{R}}
\newcommand{\mB}{\mathcal{B}}
\newcommand{\mE}{\mathcal{E}}
\newcommand{\mS}{\mathcal{S}}
\newcommand{\mO}{\mathcal{O}_\mathcal{H}}

\newcommand{\ve}{\varepsilon}

\newcommand{\lmt}[2]{\lim_{#1}{#2}}

\newcommand{\dist}{{\rm dist}_{\mathcal{H}}\,}
\newcommand{\mfq}{\mathfrak{q}}

\begin{document}

\title[Convergence of gradient flows]
{Convergence to equilibrium of gradient flows defined on planar curves}
\author{Matteo Novaga}
\address{Department of Mathematics\\University of Pisa\\Largo Bruno Pontecorvo 5\\ 36127 Pisa, Italy\\}
\email{novaga@dm.unipi.it}
\author{Shinya Okabe}
\address{Mathematical Institute\\Tohoku University\\980-8578 Sendai, Japan\\}
\email{okabe@math.tohoku.ac.jp}

\begin{abstract}
We consider the evolution of open planar curves by the steepest descent flow of a geometric functional,
under different boundary conditions.
We prove that, if any set of stationary solutions with fixed energy is finite, 
then a solution of the flow converges to a stationary solution as time goes to infinity.
We also present a few applications of this result. 
\end{abstract}

\maketitle

\tableofcontents

\section{Introduction}
The steepest descent flow for the total squared curvature defined on curves has been widely studied in the literature. 
By virtue of a smoothing effect of the functional, there are various results such that 
the flow has a smooth solution for all times and subconverges to a (possibly nonunique) stationary solution 
(e.g. \cite{BGN}, \cite{b-m-n}, \cite{acqua-pozzi}--\cite{polden}, \cite{wen-1}--\cite{wheeler}). 
However there are few results proving the full convergence of solutions to a stationary state, see
for instance \cite{koiso,okabe-1,okabe-2,polden,wen-1,wen-2}. For the case of closed curves in $\R^n$, 
it has been recently proved in \cite{wheeler}
that the $L^2$-gradient flow for a generalized Helfrich functional 
has a solution for any time, and the solution converges 
to an equilibrium. This strengthens a result obtained in \cite{dziuk} and \cite{polden}. 

In \cite{koiso,okabe-1,okabe-2} convergence is proved with the aid of an additional constraint, 
the so-called inextensible condition, while
in \cite{polden, wen-1,wen-2} it follows from the uniqueness of the equilibrium state. 
The purpose of this paper is to prove convergence under a weaker condition, namely that there are only finitely many equilibrium states
at each prescribed energy level. 

The structure of the paper is the following:
first, in Section \ref{Hilbert}, we present our method in the case of abstract gradient flows
in Hilbert spaces; 
then, in Section \ref{main-p}, we adapt it to the evolution of planar curves corresponding 
to the gradient flow of a geometric functional; finally,
in Section \ref{application}, we apply it to the equation
\begin{align} \label{s-s-i}
\pd_t \gm= (-2 \pd^2_s \vk - \vk^3 + \lm^2 \vk) \bn, 
\end{align}
where $\lm\in\R\setminus\{0\}$ and $\vk$, $\bn$ are respectively the scalar curvature and the unit normal,
under typical boundary conditions. More precisely, we discuss in detail the boundary conditions: 
\begin{enumerate}
\item[(i)] $\gm(0,t)=(0,0)$, \,\, $\gm(1,t)=(R,0)$, \,\, $\gm_s(0,t)= \tau_0$, \,\, $\gm_s(1,t)= \tau_1$
\item[(ii)] $\gm(0,t)=(0,0)$, \,\, $\gm(1,t)=(R,0)$, \,\, $\vk(0,t)= \vk(1,t)=\va$  
\end{enumerate}
where $\tau_0$, $\tau_1 \in \R^2$ are given constant unit vectors and $\va \in \R$ is a prescribed constant. 
Condition (i) is usually called {\it clamped boundary condition} (see \cite{lin}), 
and (ii) is referred to as {\it symmetric Navier boundary condition} (see \cite{BGN, dec-gru}).

Eventually, Appendix \ref{appendix} is concerned with the analyticity of certain functions 
which play an important r\^ole 
in the proof of the convergence result, while
in Appendix \ref{appendix-b} we prove the long time existence of smooth solutions to 
\eqref{s-s-i} under the boundary condition (ii).


\section{Gradient flows in Hilbert spaces} \label{Hilbert}
Let $\mH$ be a Hilbert space and $F: \mH \to [0, +\infty]$ be a functional satisfying
the following assumptions: 
\begin{align}
& \{ u \in \mH \mid F(u) \leq C \} \quad \text{is compact for any} \quad C \in [0, +\infty); \label{F-a1}  
\\
& \| \pd^0 F(u) \|^2_{\mH} \quad \text{is lower semi-continuous in } \mH. \label{F-a2}  
\end{align}
Here $\pd^0 F$ denotes the canonical element of the sub-differential $\pd F$, defined as 
the (unique) element of $\pd F$ of minimal norm in $\mH$ (see \cite{a-g-s}).

\begin{rem}
\begin{enumerate}
\item[{\rm (1)}] The condition \eqref{F-a1} implies that $F$ is lower semi continuous {\rm (see {\rm \cite{struwe}})}. 
\item[{\rm (2)}] If $F$ admits the decomposition $F= F_1 + F_2$, where $F_1 : \mH \to [0, +\infty]$ is lower semi continuous 
and convex functional, whereas $F_2 : \mH \to [0, \infty)$ is of class $C^1$, 
then $F$ satisfies \eqref{F-a2} {\rm (see {\rm \cite{a-g-s}})}.  
\end{enumerate}
\end{rem}

Let $u(t) \in H^1((0, +\infty); \mH)$ be a function satisfying the following: 
\begin{equation}
u_t = - \pd^0 F(u) \quad \text{for all} \quad t >0; \label{u-a3}.
\end{equation}
It is shown in \cite[Lemma 3.3]{brezis} that the function $t\mapsto F(u(t))$ is absolutely continuous, 
and its derivative satisfies
\begin{equation}\label{u-a4}
\dfrac{d}{dt} F(u(t))= - \|u_t\|^2_\mH
= - \| \pd^0 F(u) \|^2_\mH.  
\end{equation}
\begin{lem} \label{to-equilibrium}
Let $F : \mH \to [0, +\infty]$  satisfy \eqref{F-a1}-\eqref{F-a2}, and  
Let $\{ t_j \}_j$ be a monotone increasing sequence with $\nf{j \in \N}{(t_{j+1} - t_j)}>0$. 
Assume that $u \in H^1((0,+\infty);\mH)$ satisfies \eqref{u-a3}. 
Then, for any $0<\ve \le \nf{j \in \N}{(t_{j+1} - t_j)}$, there exists a sequence $\{ t^\ve_j \}_j$ with 
$t^{\ve}_j \in (t_j, t_j + \ve)$ such that 
\begin{align*}
\| \pd^0 F(u(t^\ve_j)) \|_\mH \to 0 \quad \text{as} \quad j \to \infty. 
\end{align*}
\end{lem}
\begin{proof}
From \eqref{u-a4} it follows that  
\begin{align} \label{G-energy-bd}
\int^\infty_0 \| \pd^0 F(u(t)) \|^2_\mH \, dt 
 = - \int^\infty_0 \dfrac{d}{dt} F(u(t)) \, dt 
 = \Bigm[ F(u(t)) \Bigm]^{t=0}_{t=\infty} < \infty, 
\end{align}
whence 
\begin{align} \label{sum-A-1}
\sum^\infty_{j=1} \int^{t_j + \ve}_{t_j} \| \pd^0 F(u(t)) \|^2_\mH \, dt < +\infty. 
\end{align}
In particular, there holds
\begin{align} \label{lim-A-1}
\lim_{j\to \infty}\int^{t_j + \ve}_{t_j} \| \pd^0 F(u(t)) \|^2_\mH \, dt = 0,
\end{align}
which implies the thesis. 
\end{proof}

Let $\mS = \{ u \in \mH \mid \pd^0 F(u)=0 \}$ the set of all stationary solutions to \eqref{u-a3}-\eqref{u-a4}.
We shall assume that 
\begin{align} \label{discrete}
\Sigma_A := \{ u \in \mS \mid F(u) = A \} \quad \text{is discrete in} \quad \mH,  
\end{align}
for all $A\in [0,+\infty)$.

\begin{thm} \label{general-case}
Suppose that $F : \mH \to [0, +\infty]$ satisfies \eqref{F-a1}-\eqref{F-a2}, and assume that \eqref{discrete} holds.
Let $u(t)  \in H^1((0, +\infty); \mH)$ be a solution to \eqref{u-a3}. 
Then there exists a unique function $\tilde{u} \in \mS$ such that 
\begin{align}
\| u(t) - \tilde{u} \|_\mH \to 0 \quad \text{as} \quad t \to \infty. 
\end{align}
\end{thm}
\begin{proof}
Remark that \eqref{F-a1}-\eqref{F-a2} and Lemma \ref{to-equilibrium} imply that $u(t)$ subconverges to a element of $\mS$, i.e., 
there exist a monotone sequence $\{ t_j \}$ with $t_j \to \infty$ as $j \to \infty$ and $\tilde{u} \in \mS$ 
such that $u(t_j) \to \tilde{u}$ in $\mH$ as $j \to \infty$. 
We prove Theorem \ref{general-case} by contradiction. 
Suppose not, there exist sequences $\{ t^1_j \}$, $\{ t^2_j \}$ and functions $\tilde{u}_1$, $\tilde{u}_2 \in \mS$ such that 
\begin{align} \label{H-sub-1}
u(t^1_j) \to \tilde{u}_1, \quad u(t^2_j) \to \tilde{u}_2 \quad \text{in} \quad \mH \quad \text{as} \quad j \to \infty. 
\end{align}
Put 
\begin{align*}
A = F(\tilde{u}_1)= F(\tilde{u}_2)= \lmt{t \to \infty}{F(u(t))}. 
\end{align*}
By \eqref{discrete}, the set $\Sigma_A$ is discrete in $\mH$, i.e., 
there exists a constant $\vd_A>0$ such that 
\begin{align*}
\| \tilde{u}_l - \tilde{u}_m \|_\mH > \vd_A
\end{align*}
for any $\tilde{u}_l$, $\tilde{u}_m \in \Sigma_A$. 
Let $\vd= \vd_A / 2$. Then, for any $\tilde{u}_l$, $\tilde{u}_m \in \Sigma_A$, it holds that 
\begin{align} \label{H-ing-2}
B_\vd(\tilde{u}_l) \cap B_\vd(\tilde{u}_m) = \emptyset, 
\end{align}
where 
\begin{align*}
B_\vd(\tilde{u})= \{ v \in \mH \mid \| v - \tilde{u} \|_\mH < \vd \}. 
\end{align*}
It follows from \eqref{H-sub-1} that there exists $J \in \N$ such that 
\begin{align} \label{H-ing-3}
u(t^1_j) \in B_\vd(\tilde{u}_1), \quad u(t^2_j) \in B_\vd( \tilde{u}_2)
\end{align}
for any $j \geq J$. Up to a subsequence, we may assume that it holds that 
\begin{align*}
t^1_j < t^2_j < t^1_{j+1}
\end{align*}
for any $j \geq J$. Then, by \eqref{H-ing-2}, \eqref{H-ing-3}, and the continuity of $u(t)$ in $\mH$, 
we see that  there exists a monotone increasing sequence $\{t^3_j \}_j$ such that 
\begin{align*} 
\| u(t^3_j) - \tilde{u} \|_\mH \geq \vd 
\end{align*}
for any $j \in \N$ and $\tilde{u} \in \Sigma_A$. 
Up to a subsequence we can assume $\inf_{j \in \N}{(t^3_{j+1} - t^3_j)}>0$. 
Since $u \in H^1((0,\infty) ; \mH)$ yields that $u(t)$ is uniformly continuous in $\mH$, 
for sufficiently small $\ve>0$, it holds that 
\begin{align} \label{H-away}
\| u(t) - \tilde{u} \|_\mH \geq \dfrac{\vd}{2}
\end{align}
for any $t \in [t^3_j, t^3_j+\ve]$ and $\tilde{u} \in \Sigma_A$. 

Suppose that there exists a sequence $\{ t^\ve_j \}_j$ with $t^\ve_j \in [t^3_j, t^3_j + \ve]$ such that 
$\| \pd^0 F(u(t^\ve_j)) \|_\mH \to 0$.  
Then by \eqref{F-a1} there exist $\hat{u} \in \mH$ and $\{ t^\ve_{j_k} \} \subset \{ t^\ve_j \}$ such that 
$u(t^\ve_{j_k}) \to \hat{u}$ in $\mH$. 
Although \eqref{F-a2} implies $\hat{u} \in \Sigma_A$, this contradicts \eqref{H-away}. 
Hence we observe that, for any $\{ t^\ve_j \}_j$ with $t^\ve_j \in [t^3_j, t^3_j + \ve]$, 
\begin{align*}
\| \pd^0 F(u(t^\ve_j)) \|_\mH \not \to 0 
\end{align*}
as $j \to \infty$. This contradicts Lemma \ref{to-equilibrium}, and  
completes the proof. 
\end{proof}


\section{Geometric gradient flows} \label{main-p}

In this section we apply the strategy described above to the gradient flow of a geometric functional
$\mE(\gm)$ defined on planar curves $\gm : I \to \R^2$, which we 
assume bounded from below, that is, $\inf_\gm \mE(\gm)>-\infty$. A $L^2$-gradient flow of $\mE$ is a one parameter family 
of curves $\gm : I \times [0, \infty) \to \R^2$ such that  
\begin{align} \label{g-flow}
\pd_t \gm = -\nabla \mE(\gm) 
\end{align}
and 
\begin{align} \label{energy-eq}
\dfrac{d}{dt} \mE(\gm(t)) = - \int_\gm \av{\nabla \mE(\gm(t))}^2 \, ds, 
\end{align}
where $\nabla \mE(\gm)$ denotes the Euler-Lagrange operator of $\mE(\gm)$, i.e., 
$\nabla \mE(\gm)$ satisfies 
\begin{align*}
\dfrac{d}{d \ve} \mE(\gm(\cdot) + \vp(\cdot,\ve)) \biggm|_{\ve=0} 
 = \int_\gm \nabla \mE(\gm) \cdot \vp_\ve \, ds 
\end{align*}
for any $\vp \in C^\infty((-\ve_0, \ve_0):(C^\infty_c(I))^2)$, where $\vp_\ve= \vp_\ve(\cdot, 0)$. 

Since the curves are open, in order to have uniqueness of the evolution we need to impose a 
boundary condition $\mB(\gm)=0$ on $\pd I$. Notice that \eqref{energy-eq} does not follow from
\eqref{g-flow} if the boundary condition given by $\mB$ is not {\it natural} for $\mE$, i.e., 
the flow \eqref{g-flow} with a boundary condition is not always the $L^2$-gradient flow for $\mE(\gm)$. 
Indeed, if $\gm$ satisfies \eqref{g-flow} under an unnatural boundary condition $\tilde{\mB}(\gm)=0$, 
then it can happen that \eqref{energy-eq} does not hold. Therefore we shall assume the following:

\begin{asum}[\bf Compatibility] \label{Asum-B}
The flow \eqref{g-flow} with boundary condition $\mB(\gm)=0$ is a $L^2$-gradient flow for $\mE(\gm)$. 
\end{asum}

Given a smooth curve $\gm$ we let $s\in [0,\mL(\gm)]$ be the arclength parameter defined as
\begin{align*}
s(x):=\int_0^x |\gm_x|\,dx \qquad x\in I,
\end{align*}
where $\mL(\gm)$ is the length of $\gm$
\begin{align*}
\mL(\gm):=s(1)=\int_0^1 |\gm_x|\,dx.
\end{align*}
Notice that in the arclength variable $s$ there holds $|\gm_s(s)|=1$ for all $s\in [0,\mL(\gm)]$.
Given a function $f(s)$ defined on $\gm$, we let
\begin{align*}
\|f\|_{L^\infty_\gamma}:=\sup_{s\in \mL(\gm)}|f(s)|
\qquad \Lg{f}:= \left( \int_\gm f(s)^2 \, ds \right)^\frac{1}{2}.
\end{align*}

We shall consider the initial boundary value problem: 
\begin{align} \label{g-ibp}
\begin{cases}
& \pd_t \gm = -\nabla \mE(\gm) \quad \, \text{in} \quad I \times (0, \infty), \\
& \mathcal{B}(\gm(x,t))=0 \quad \,\, \text{on} \quad \pd I \times [0, \infty), \\ 
& \gm(x,0)= \gm_0(x) \quad \text{in} \quad I,  
\end{cases}
\end{align}
where $\gm_0(x) : I \to \R^2$ is a smooth planar open curve satisfying the boundary condition $\mB(\gm_0(x))=0$ on $\pd I$. 
Regarding the solvability of \eqref{g-ibp}, we assume the following: 
\begin{asum}[\bf Regularity] \label{U-B}
There exists a smooth solution $\gm:I\times [0,+\infty)\to \R^2$ of \eqref{g-ibp}, satisfying  
\begin{align} \label{U-B-s}
\|\pd_t \gm(\cdot,t)\|_{L^\infty_{\gamma(t)}} < C \qquad {\rm and}\qquad 
\int_\gm \av{\pd_s^m \gm(t)}^2 \, ds < C  
\end{align}
for any $m \in \N$ and 
for any $t>0$, where the constant $C$ is independent of $t$. 
Moreover, $\Lg{\nabla \mE(\gm)}$ is continuous in $\gm$
with respect to the $C^\infty$-topology.
\end{asum}
Notice that, as the functional $\mE$ is bounded from below, then \eqref{energy-eq} implies the estimate 
\begin{equation}\label{esso}
\int_0^{+\infty}\|\pd_t \gm(\cdot,t)\|^2_{L^2_{\gamma(t)}}\le \mE(\gamma_0)-\inf \mE
\end{equation}
for any solution $\gm$ of \eqref{g-ibp}.

Under an additional assumption on $(\mE,\mB)$, 
we shall prove that a solution of \eqref{g-ibp} converges to a stationary solution as $t \to +\infty$. 
Let $\mS$ be a set of all stationary solutions of \eqref{g-ibp}, i.e., the smooth curves $\tilde{\gm}$ satisfying 
\begin{align} \label{g-stationary}
\begin{cases}
& \nabla \mE(\tilde{\gm}(x)) = 0 \quad \text{in} \quad I, \\
& \mathcal{B}(\tilde{\gm}(x))=0 \quad \,\,\,\, \text{on} \quad \pd I. \\   
\end{cases}
\end{align}
For each $A \in \R$, we define the subset of $\mS$  
\begin{align*} 
\Sigma_A := \{ \tilde{\gm} \in \mS \mid \mE(\tilde{\gm}) = A \}. 
\end{align*}
We shall assume the following:
\begin{asum} \label{A-B}
$\Sigma_A$ is finite for any $A \in \R$. 
\end{asum}
We can now state our main result.
\begin{thm} \label{main-thm}
Let $\gm(x,t) : I \times [0,\infty) \to \R^2$ be a solution of \eqref{g-ibp}, and
suppose that Assumptions \ref{Asum-B}, \ref{U-B} and \ref{A-B} hold. 
Then, there exists a smooth curve $\tilde{\gm}:I\to \R^2$ satisfying \eqref{g-stationary} and such that 
\begin{align*}
\gm(\cdot,t) \to \tilde{\gm}(\cdot) \quad \text{as} \quad t \to \infty 
\end{align*}
in the $C^\infty$-topology. 
\end{thm}
We start with the analog of Lemma \ref{to-equilibrium}.
\begin{lem} \label{l-to-stationary}
Let $\{ t_j \}^\infty_{j=1}$ be a monotone increasing sequence with $\nf{j \in \N}{(t_{j+1} - t_j)}\ge 0$. 
Then, for any $0 < \ve \le \nf{j \in \N}{(t_{j+1} - t_j)}$, there exists a sequence $\{ t^\ve_j \}_j$ with 
$t^{\ve}_j \in (t_j, t_j + \ve)$ such that 
\begin{align*}
\Lg{\nabla \mE(\gm(t^\ve_j))} \to 0 \quad \text{as} \quad j \to \infty. 
\end{align*}
\end{lem}
\begin{proof}
Let fix $0 < \ve < \nf{j \in \N}{(t_{j+1} - t_j)}$ arbitrarily. 
Recall that 
\begin{align*}
\int^\infty_0 \Lg{\nabla \mE(\gm(t))}^2 \, dt 
 = - \int^\infty_0 \dfrac{d}{dt} \mE(\gm(t)) \, dt 
 = \Bigm[ \mE(\gm(t)) \Bigm]^{t=0}_{t=\infty} < +\infty,  
\end{align*}
so that we have 
\begin{align} \label{A-1}
\sum^\infty_{j=1} \int^{t_j + \ve}_{t_j} \Lg{\nabla \mE(\gm(t))}^2 \, dt < \infty, 
\end{align}
which implies
\begin{align} \label{limA-1}
\lim_{j\to\infty} \int^{t_j + \ve}_{t_j} \Lg{\nabla \mE(\gm(t))}^2 \, dt =0. 
\end{align}
The thesis follows directly from \eqref{limA-1}.
\end{proof}

We now prove Theorem \ref{main-thm}.  

\begin{proof}
To begin with, remark that Assumption \ref{U-B} and Lemma \ref{l-to-stationary} imply that the solution $\gm$ 
subconverges to a stationary solution $\tilde{\gm}$ as $t\to \infty$. 
Indeed, by Lemma \ref{l-to-stationary}, one can find a sequence $\{ t_j \}$ with $t_j \to \infty$ such that 
\begin{align} \label{sub-c-ing-1}
\Lg{\nabla \mE(\gm(t_j))} \to 0 \quad \text{as} \quad t_j \to \infty. 
\end{align} 
Since Assumption \ref{U-B} allows us to apply Arzel\`a-Ascoli's theorem to the family of planar 
open curves $\gm(t_j)$, we see that there exists a subsequence $\{ t_{j_k} \} \subset \{ t_j \}$ 
such that 
\begin{align} \label{sub-c-ing-2}
\gm(t_{j_k}) \to \tilde{\gm} \quad \text{as} \quad t_{j_k} \to \infty
\end{align}
in the $C^\infty$-topology. Combining \eqref{sub-c-ing-1} with the definition of the $L^2$-gradient flow 
\eqref{g-flow}-\eqref{energy-eq}, we observe that the limit $\tilde{\gm}$ is independent of $t$ and satisfies 
\begin{align}
\nabla \mE(\tilde{\gm})=0 \quad \text{on} \quad I. 
\end{align}
 
We shall prove Theorem \ref{main-thm} by contradiction. 
Suppose not, there exist sequences $\{t^1_j \}_j$, $\{ t^2_j \}_j$ and 
stationary solutions $\tilde{\gm}_1$, $\tilde{\gm}_2 \in \mS$ such that 
\begin{align} \label{asum-1} 
\gm(t^1_j) \to \tilde{\gm}_1, \quad \gm(t^2_j) \to \tilde{\gm}_2 
\end{align} 
as $j \to \infty$. We may assume that $\{ t^1_j \}_j$ and $\{ t^2_ j \}_j$ are monotone increasing sequences. 
Let 
\begin{align*}
A = \mE(\tilde{\gm}_1)= \mE(\tilde{\gm}_2). 
\end{align*}
Thanks to Assumption \ref{A-B}, the set $\Sigma_A$ is finite. 
On the other hand, for each curves $\tilde{\gm}_n$, $\tilde{\gm}_m \in \Sigma_A$, 
there exists a constant $\vd_{nm}>0$ such that 
\begin{align*}
\dist(\tilde{\gm}_n, \tilde{\gm}_m) > \vd_{nm}, 
\end{align*}
where $\dist(\cdot,\cdot)$ denotes the Hausdorff distance defined as follows: 
\begin{align*}
\dist(\gm, \Gm)  
 = \mx{}{ \left\{  \su{u \in Im(\gm)}{ \nf{v \in Im(\Gm)}{ \av{u-v} }}, 
                    \su{v \in Im(\Gm)}{ \nf{u \in Im(\gm)}{ \av{u-v} } } \right\}}. 
\end{align*}
Since $\Sigma_A$ is finite, there exists a constant $\vd_*>0$ such that 
\begin{align} \label{ing-1}
\mn{\tilde{\gm}_n, \tilde{\gm}_m \in \Sigma_A}{ \dist(\tilde{\gm}_n, \tilde{\gm}_m) } = \vd_*. 
\end{align}
Let $\vd= \vd_* / 2$. Then, for any $\tilde{\gm}_n$, $\tilde{\gm}_m \in \Sigma_A$, it holds that 
\begin{align} \label{ing-2}
\mO(\tilde{\gm}_n, \vd) \cap \mO(\tilde{\gm}_m, \vd) = \emptyset, 
\end{align}
where 
\begin{align*}
\mO(\tilde{\gm}, \vd)= \{ \gm \mid \dist(\tilde{\gm}, \gm) < \vd \}. 
\end{align*}
It follows from \eqref{asum-1} that there exists $J \in \N$ such that 
\begin{align} \label{ing-3}
\gm(t^1_j) \in \mO(\tilde{\gm}_1, \vd), \quad \gm(t^2_j) \in \mO( \tilde{\gm}_2, \vd)
\end{align}
for any $j \geq J$. Up to a subsequence, we may assume that it holds that 
\begin{align*}
t^1_j < t^2_j < t^1_{j+1}
\end{align*}
for any $j \geq J$. Then, by \eqref{ing-1}, \eqref{ing-2}, \eqref{ing-3}, and the continuity of $\dist$, 
we see that  there exists a monotone increasing sequence $\{t^3_j \}_j$ such that 
\begin{align} \label{clm-1}
d(t^3_j) > \vd 
\end{align}
for any $j \in \N$, where $d(t) := \min_{\tilde{\gm} \in \Sigma_A}{ \dist(\gm(t), \tilde{\gm}) }$. 
Up to a subsequence, we may also assume that $\inf_{j \in \N}{(t^3_{j+1} - t^3_j)} >0$. 

Here we claim that the function $d(t)$ is Lipschitz continuous on $(0,+\infty)$. 
Remark that \eqref{U-B-s} gives us that there exists a constant $C>0$ such that 
\begin{align} \label{bd-gm_t}
\sup_{x \in I}{\av{\pd_t \gm(x,t)}} < C
\end{align}
for any $t>0$. Let $x$, $y \in I$ fix arbitrarily. 
Then the fact \eqref{bd-gm_t} yields that 
\begin{align*}
\av{\av{\gm(x,t_1) - \tilde{\gm}(y)} - \av{\gm(x,t_2) - \tilde{\gm}(y)} } 
 &\leq \av{\gm(x,t_1) - \gm(x,t_2)} \\
 &\leq \int^{t_1}_{t_2} \av{\pd_t \gm(x,t)} \, dt
 < C \av{t_1 - t_2}. 
\end{align*}
Combining the estimate with the definition of the Hausdorff distance, we obtain  
\begin{align*}
\av{d(t_1) - d(t_2)} < C \av{t_1 - t_2}.    
\end{align*}
Thus the function $d(t)$ is $C$-Lipschitz, in particular uniform continuous, on $(0, +\infty)$. 
Then it follows from \eqref{clm-1} that there exists $0 < \ve < \inf_{j \in \N}{(t^3_{j+1} - t^3_j)}$ such  that 
\begin{align} \label{ing-4}
\dist (\gm(t), \Sigma_A) \geq \dfrac{\vd}{2} 
\end{align}
for any $j \in \N$ and any $t \in [t^3_j, t^3_j + \ve]$. 
The inequality \eqref{ing-4} implies that, for any $\{ t^\ve_j \}_j$ with $t^\ve_j \in [t^3_j, t^3_j + \ve]$, 
$\gm(t^\ve_j)$ does not converges to any stationary solution as $j \to \infty$. 
However, by Assumption \ref{U-B} and Lemma \ref{l-to-stationary}, we can find a sequence 
$\{ \tilde{t}_j \}$ with $\tilde{t}_j \in (t^3_j, t^3_j + \ve]$ such that $\gm(\tilde{t}_j)$ converges to a stationary solution,
which gives a contradiction. 
\end{proof}


\section{Application to the  shortening-straightening flow} \label{application}
In this section, we apply Theorem \ref{main-thm} to the geometric equation  
\begin{align} \label{s-s}
\pd_t \gm = (-2 \pd^2_s \vk - \vk^3 + \lm^2 \vk) \bn, 
\end{align}
where $\vk$ and $\bn$ denote respectively the scalar curvature and the unit normal vector 
with the direction of the curvature, and $\lm$ is a non-zero constant. 
Throughout the section we assume that $\gm(x,t) : I \times [0,\infty) \to \R^2$ are fixed at the boundary, i.e., 
\begin{align} \label{fixed-p}
\gm(0,t)= (0,0), \quad \gm(1,t)=(R,0) \quad \text{on} \quad [0,\infty), 
\end{align}
where $R>0$ is a given constant. 

Here we prepare several notations. 
In what follows let us set 
\begin{align*}
f(\vk)= \vk^3 - \lm^2 \vk. 
\end{align*}
From the Euler-Lagrange equation
\begin{align*}
2 \pd^2_s \vk + \vk^3 - \lm^2 \vk=0, 
\end{align*}
we obtain the relation 
\begin{align} \label{f-integ}
\left( \dfrac{d \vk}{ds} \right)^2 + F(\vk) = E, 
\end{align}
where $E$ is an arbitral constant and $F'=f$, i.e., $F$ is given by 
\begin{align*}
F(\vk)= \dfrac{1}{4} \vk^4 - \dfrac{\lm^2}{2} \vk^2. 
\end{align*}
Let $\vk_M(E)$ and $\vk_m(E)$ be solutions of $F(\vk)=E$ as follows: 
\begin{align*}
& \vk_M(E) = \sqrt{\lm^2 + \sqrt{\lm^4+4E}} \quad \quad \,\,\,\, \text{for} \quad E \in (-\tfrac{\lm^4}{4}, \infty) , \\
& \vk_m(E) = 
\begin{cases}
& - \vk_M(E) \quad \qquad \quad \,\,\,\,\, \text{for} \quad E \in (0, \infty), \\
& \sqrt{\lm^2 - \sqrt{\lm^4+4E}} \quad \text{for} \quad E \in (- \tfrac{\lm^4}{4}, 0]. 
\end{cases}
\end{align*}
If there is no fear of confusion, we write $\vk_M$ and $\vk_m$ instead of $\vk_M(E)$ and $\vk_m(E)$. 
Let us set 
\begin{align*} 
L(E)= 2 \int^{\vk_M(E)}_{\vk_m(E)} \dfrac{d \vk}{\sqrt{E - F(\vk)}}. 
\end{align*}
\begin{lem} \label{l-p-E-div}
Let $(\vk, E)$ be a pair satisfying 
\begin{align*}
\begin{cases}
\left( \dfrac{d \vk}{ds} \right)^2 + F(\vk) = E, \\
\vk(0)=0. 
\end{cases}
\end{align*}  
Then it holds that 
\begin{align}
\int^{L(E)}_0 \vk(s)^2 \, ds \to \infty \quad \text{as} \quad E \to \infty.  \label{p-E-div}
\end{align}
\end{lem}
\begin{proof}
Since it holds that 
\begin{align} \label{property-Le}
L(E)= 4 \int^{\vk_M(E)}_0 \dfrac{d \vk}{\sqrt{E - F(\vk)}} \quad \text{for} \quad E>0, 
\end{align}
it is sufficient to prove that 
\begin{align} \label{pre-p-E-div} 
\int^{\frac{L(E)}{4}}_0 \vk(s)^2 \, ds \to \infty \quad \text{as} \quad E \to \infty. 
\end{align}
Since $\vk_M(E) \to \infty$ as $E \to \infty$,  it holds  that 
\begin{align*}
\sqrt{2} \av{\lm} < \vk_M(E) 
\end{align*}
for sufficiently large $E$, where $\sqrt{2} \av{\lm}$ is a solution of $F(\vk)=0$. 
Then we have  
\begin{align} \label{L1}
\int_0^{\frac{\vk_M}{2}} \dfrac{d \vk}{\sqrt{E-F(\vk)}} 
 &= \int_0^{\sqrt{2} \av{\lm}} \dfrac{d \vk}{\sqrt{E - F(\vk)}} 
        + \int_{\sqrt{2} \av{\lm}}^{\frac{\vk_M}{2}} \dfrac{d \vk}{\sqrt{E-F(\vk)}} \\
 &< \dfrac{\sqrt{2} \av{\lm}}{\sqrt{E}} 
        + \dfrac{\vk_M/2 - \sqrt{2} \av{\lm} }{\sqrt{E - F(\vk_*)}}, \notag
\end{align}
where $\vk_* \in (\sqrt{2} \av{\lm}, \vk_M/2)$. 
On the other hand, it holds that  
\begin{align} \label{L2}
\int_{\frac{\vk_M}{2}}^{\vk_M} \dfrac{d \vk}{\sqrt{E-F(\vk)}} 
 > \dfrac{\vk_M/2}{\sqrt{E - F(\vk_*)}}.  
\end{align}
If $E>1$, then we find   
\begin{align} \label{L2-L1}
\dfrac{\vk_M/2}{\sqrt{E - F(\vk_*)}} 
  - \left\{ \dfrac{\sqrt{2} \av{\lm}}{\sqrt{E}} 
        + \dfrac{\vk_M/2 - \sqrt{2} \av{\lm} }{\sqrt{E - F(\vk_*)}} \right\}  > 0. 
\end{align}
Combining \eqref{L1}-\eqref{L2} with \eqref{L2-L1}, we observe that 
\begin{align*}
\int_{\frac{\vk_M}{2}}^{\vk_M} \dfrac{d \vk}{\sqrt{E-F(\vk)}}  
 > \int_0^{\frac{\vk_M}{2}} \dfrac{d \vk}{\sqrt{E-F(\vk)}}
\end{align*}
for sufficiently large $E$. Set 
\begin{align*}
\mathfrak{L}_1= \int_0^{\frac{\vk_M}{2}} \dfrac{d \vk}{\sqrt{E-F(\vk)}}, \qquad
\mathfrak{L}_2= \int_{\frac{\vk_M}{2}}^{\vk_M} \dfrac{d \vk}{\sqrt{E-F(\vk)}}. 
\end{align*}
By virtue of \eqref{property-Le}, we see that 
\begin{align} \label{lower-est}
\int^{\frac{L(E)}{4}}_0 \vk(s)^2 \, ds 
 > \int^{\frac{L(E)}{4}}_{\mathfrak{L}_1} \vk(s)^2 \, ds 
 > \dfrac{\mathfrak{L}_2 {\vk_M}^2}{4}  
 > \dfrac{{\vk_M}^3}{8 \sqrt{E}}. 
\end{align}
Here we used \eqref{L2}. Since it holds that 
\begin{align*}
\lmt{E \to \infty}{\dfrac{{\vk_M(E)}^2}{2 \sqrt{E}} } 
 = \lmt{E \to \infty}{\dfrac{\lm^2 + \sqrt{\lm^4 + 4E}}{2 \sqrt{E}} }
 \to 1, 
\end{align*}
the estimate \eqref{lower-est} implies \eqref{pre-p-E-div}. 
\end{proof}

\subsection{Clamped boundary condition} \label{section-cp}

Recently C.-C. Lin considered a motion of open curves in $\R^n$ with boundary points fixed. 
Although he considered the problem for any $n \geq 2$ (\cite{lin}), we restrict the dimension $n=2$. 
The motion is governed by the geometric evolution equation \eqref{s-s} with the boundary condition 
\eqref{fixed-p} and 
\begin{align} \label{clamped-cond}
\gm_s(0,t)= \tau_0, \quad \gm_s(1,t)= \tau_1, 
\end{align}
where  $\tau_0$, $\tau_1 \in \R^2$ are prescribed unit vectors. 
The boundary condition \eqref{fixed-p}-\eqref{clamped-cond} is called 
the {\it clamped boundary condition}. 
One can verify that Assumption \ref{Asum-B} holds, i.e., 
the flow \eqref{s-s} with the clamped boundary condition is 
a $L^2$-gradient flow for the functional 
\begin{align} \label{mtsc}
\mE_\lm(\gm)= \int_{\gm} ( \vk^2 + \lm^2 ) \, ds.    
\end{align}
The functional is well known as the modified total squared curvature. 

Let $\gm_0 : I \to \R^2$ be a smooth planar open curve satisfying the following: 
\begin{align*}
\gm_0(0)= (0,0),\,\, \gm_0(1)= (R, 0),\,\, \gm_{0s}(0)= \tau_0,\,\, \gm_{0s}(1)= \tau_1. 
\end{align*}
For such curve $\gm_0$, we consider the following initial boundary value problem: 
\begin{align} \label{clamped}
\begin{cases}
& \pd_t \gm = (-2 \pd^2_s \vk - \vk^3 + \lm^2 \vk) \bn \qquad \quad 
\text{in} \quad I \times [0, \infty), \\
& \gm(0,t)= (0,0),\,\,\, \gm(1,t)= (R, 0),\\ 
& \gm_s(0,t)= \tau_0,\,\,\, \gm_s(1,t)= \tau_1 \qquad \quad \,\,\,\,\,
\text{in} \quad [0, \infty), \\ 
& \gm(x,0)= \gm_0(x) \qquad \qquad \qquad \qquad \,\,\,\, 
\text{in} \quad I.   
\end{cases}
\end{align}
The purpose of this subsection is to prove a convergence of a solution of \eqref{clamped} to an 
equilibrium as $t \to \infty$. 
Regarding the problem \eqref{clamped}, C.-C. Lin obtained the following result: 
\begin{prop}{\rm (\cite{lin})} \label{c-c-lin}
For any prescribed constant $\lm \neq 0$ and smooth initial curve $\gm_0$ with finite length, 
there exists a global smooth solution $\gm$ of \eqref{clamped}. Moreover, after reparametrization 
by arc length, the family of curves $\{ \gm(t) \}$ subconverges to $\gm_\infty$, which is an equilibrium 
of the energy functional \eqref{mtsc}.  
\end{prop}
It follows from the proof of Proposition \ref{c-c-lin} that Assumption \ref{U-B} holds. 

Let $\mS$ be a set of all stationary solutions of \eqref{clamped}, i.e., all open curves satisfying 
\begin{align} \label{st-cp}
\begin{cases}
& 2 \pd^2_s \vk + \vk^3 - \lm^2 \vk =0 \quad \text{in} \quad I, \\
& \gm(0)= (0,0),\,\, \gm(1)= (R, 0),\,\, \gm_s(0)= \tau_0,\,\, \gm_s(1)= \tau_1.   
\end{cases}
\end{align}
We denote $\Sigma_A$ a subset of $\mS$ defined by 
\begin{align*}
\Sigma_A= \{ \tilde{\gm} \in \mS \mid \mE_\lm(\tilde{\gm}) = A \}. 
\end{align*}
In order to apply Theorem \ref{main-thm} to the problem \eqref{clamped}, 
we prove that Assumption \ref{A-B} holds, i.e., the set $\Sigma_A$ is finite for any $A \in \R$. 

\begin{lem} \label{cp-finite}
The set $\Sigma_A$ is finite for each $A \in \R$. 
\end{lem}
\begin{proof}
Suppose not, there exist a constant $A$ and a sequence of planar open curves 
$\{ \gm_n \}^\infty_{n=1} \subset \Sigma_A$. 
Let fix a family of planar curves $\gm(s, E)$ such that 
\begin{align*}
\left( \dfrac{d \vk}{ds} \right)^2 + F(\vk)= E, \quad \gm(0,E)=(0,0), \quad \gm_s(0,E)= \tau_0,  
\end{align*}
and 
\begin{align*}
\gm(s, E_n)= \gm_n \quad \text{on} \quad [0, \mL(\gm_n)],  
\end{align*}
where $\mL(\gm_n)$ denotes the length of $\gm_n$. 
Remark that $\gm(s,E)$ is analytic in $s$ and $E$ on $\R \times (-\lm^4/4, \infty)$. 
In particular, letting $s_n= \mL(\gm_n)$, we have 
\begin{align} \label{at_s_n}
\gm(s_n,E_n)=(R,0), \quad 
\gm_s(s_n,E_n)=\tau_1. 
\end{align}
If $E_n \to \infty$ as $n \to \infty$, then Lemma 4.4 implies 
\begin{align} \label{period-shrink}
L(E_n) \to 0 \quad \text{as} \quad n \to \infty
\end{align}
and Lemma 3.1 yields that 
\begin{align} \label{tsc_n_div}
\int^{L(E_n)}_0 \vk^2_n \, ds \to \infty \quad \text{as} \quad  n \to \infty,  
\end{align}
where $\vk_n= \vk(s, E_n)$. 
Although \eqref{period-shrink}-\eqref{tsc_n_div} yields that $\mE_\lm(\gm(\cdot, E_n)) \to \infty$ as $n \to \infty$, 
this contradicts $\mE_\lm(\gm(\cdot, E_n)) = A$. 
Thus there exists a constant $E^*$ such that $E_n < E^*$, i.e., $\{E_n \}^\infty_{n=1}$ is bounded sequence. 
Moreover the fact $\{ \gm_n \}^\infty_{n=1} \subset \Sigma_A$ implies $R \leq s_n \leq A/\lm^2$, 
i.e., $\{ s_n \}^\infty_{n=1}$ is also bounded sequence.   
Hence there exist subsequences $\{E_{n_j} \}^\infty_{j=1} \subset \{E_n \}^\infty_{n=1}$ and 
$\{s_{n_j} \}^\infty_{j=1} \subset \{s_n \}^\infty_{n=1}$ and constants $E_\infty$ and $s_\infty$ 
such that $E_{n_j} \to E_\infty$ and $s_{n_j} \to s_\infty$ as $j \to \infty$. 
In the following we write $\{E_n \}^\infty_{n=1}$ and $\{s_n \}^\infty_{n=1}$ instead of 
$\{E_{n_j} \}^\infty_{j=1}$ and $\{s_{n_j} \}^\infty_{j=1}$ for short. 

We prove that there exist a neighborhood $U$ of $E_\infty$ and a function $s : U \to \R$ such that, 
for any $E \in U$, 
\begin{align} \label{s-func}
\gm(s(E),E)=(R,0), \quad 
\gm_s(s(E),E)=\tau_1. 
\end{align}
If $\tau_1 \cdot {\bm e}_1 \neq 0$, then we define a function $\Phi : \R \times (-\lm^4/4, \infty) \to \R$ as 
$\Phi(s,E)= \gm_1(s,E)$, where ${\bm e}_1= (1,0)$ and $\gm=(\gm_1, \gm_2)$. 
Since $\Phi(s_\infty, E_\infty)=R$ and $\Phi_s(s_\infty, E_\infty)=\tau_1 \cdot {\bm e}_1 \neq 0$, 
the implicit function theorem yields that there exist a neighborhood $U$ of $E_\infty$ and 
a function $s : U \to \R$ such that, for any $E \in U$,  
\begin{align} \label{s-func-2}
\gm_1(s(E), E)=R. 
\end{align}
It follows from \eqref{at_s_n} and \eqref{s-func-2} that $s(E_n)= s_n$ holds for any $n \in \N$. 
Moreover the analyticity of $\gm$ implies that $s(E)$ is analytic on $U$.  
Combining the analyticity of $s(E)$ with 
\begin{align*}
\gm(s(E_n), E_n)= (R,0), \quad \gm_s(s(E_n), E_n)= \tau_1, 
\end{align*}
we observe that $s(E)$ satisfies \eqref{s-func} on $U$. 
If $\tau_1 \cdot {\bm e}_1 = 0$, then it is sufficient to define a function $\Phi(s,E)$ as $\Phi(s,E)= \gm_2(s,E)$. 

Let us define a function $d : (-\lm^4/4, \infty) \to \R$ as 
\begin{align*}
d(E)= \min_{s \in \R}{ \av{\gm(s,E) - (R,0)}^2 }. 
\end{align*}
Remark that the function $d(E)$ is analytic and $d(E)=0$ on $U$. 
We claim that $d(E)$ is analytic on $(-\lm^4/4, \infty)$. 
Suppose that there exists a maximal open set $V \supset U$ such that $d(E)$ is analytic on $V$. 
Then we see that $d(E)>0$ in $\pd V$. 
For, if $d(E)=0$ in $\pd V$, then the similar argument as above yields that $d(E)$ is analytic on 
a neighborhood of $\pd V$. This contradicts that $V$ is maximal. 
On the other hand, since $d(E)$ is analytic and $d(E)=0$ on $U \subset V$, 
we observe that $d(E)=0$ on $V$. 
Therefore $d(E)$ is analytic on $(-\lm^4/4, \infty)$. 

Since $d(E)=0$ on $U$, the analyticity yields that $d(E)=0$ for any $E \in (-\lm^4/4, \infty)$. 
Thus there exists an extension $s(E)$ such that 
\begin{align*}
\gm(s(E), E)=(R,0) \quad \text{for all} \quad E \in (-\lm^4/4, \infty),  
\end{align*}
where we still denote the extension as $s(E)$, for short. 

We claim that $s(E)$ is analytic on $U \cup (E_\infty, \infty)$. 
Suppose not, there exists a constant $\tilde{E}$ such that $s(E)$ is not extended analytically for $E \geq \tilde{E}$. 
Then it holds that 
\begin{align} \label{s-blow-up}
s(E) \to \infty \quad \text{as} \quad E \nearrow \tilde{E}. 
\end{align}
Since 
\begin{align*}
\mE_\lm(\gm(\cdot, E))= \int^{s(E)}_0 \vk(s, E)\, ds + \lm^2 s(E) \quad \text{for any} \quad E \in U \cup (E_\infty, \tilde{E}), 
\end{align*}
\eqref{s-blow-up} is equivalent to 
\begin{align*}
\mE_\lm(\gm(\cdot, E)) \to \infty \quad \text{as} \quad E \nearrow \tilde{E}. 
\end{align*}
This contradicts that $\mE_\lm(\gm(\cdot, E)) =A$ for any $E \in U \cup (E_\infty, \tilde{E})$. 
Therefore we see that $s(E)$ is extended analytically on $U \cup (E_\infty, \infty)$. 

We now obtain a contradiction. 
Since the analyticity of $s(E)$ implies that \eqref{s-func} holds for all $E \in U \cup (E_\infty, \infty)$, 
it follows that $\gm(s, E) \in \Sigma_A$ for all $E \in U \cup (E_\infty, \infty)$, i.e., 
\begin{align}\label{E-iden}
\mE_\lm(\gm(\cdot, E))= A \quad \text{for any} \quad E \in U \cup (E_\infty, \infty). 
\end{align}
However \eqref{E-iden} contradicts that 
\begin{align*}
\mE_\lm(\gm(\cdot, E)) \to \infty \quad \text{as} \quad E \to \infty. 
\end{align*}
We complete the proof.  
\end{proof}

Lemma \ref{cp-finite} implies that one can apply Theorem \ref{main-thm} to \eqref{clamped}. 
Then the following result is proved. 
\begin{thm} \label{appli-1}
Let $\lm \neq 0$. Let $\gm$ be a smooth solution of \eqref{clamped} obtained by Proposition \ref{c-c-lin}. 
Then, as $t \to \infty$, the solution $\gm$ converges to a solution of \eqref{st-cp} in the $C^\infty$-topology. 
\end{thm}


\subsection{Zero curvature boundary condition} \label{DBC}
In this subsection, we impose that the curvature of $\gm(x,t)$ is zero at the boundary of $I$, i.e.
\begin{align} \label{D-curvature}
\vk(0,t)= \vk(1,t)=0 \quad \text{in} \quad (0,\infty).  
\end{align}
We shall consider the initial value problem for \eqref{s-s} with the boundary conditions 
\eqref{fixed-p}-\eqref{D-curvature} 
\begin{align} \label{n-o}
\begin{cases}
& \pd_t \gm = (-2 \pd^2_s \vk - \vk^3 + \lm^2 \vk) \bn \qquad \quad  
\text{in} \quad I \times [0,\infty), \\
& \gm(0,t)= (0,0),\,\, \gm(1,t)= (R, 0),\\
& \vk(0,t)= \vk(1,t)=0 \qquad \qquad \qquad \,\,\,\,
\text{in} \quad [0,\infty), \\ 
& \gm(x,0)= \gm_0(x) \qquad \qquad \qquad \qquad \,\,\,\,
\text{in} \quad I. 
\end{cases}
\end{align}
Remark that $\gm_0$ is a smooth planar curve satisfying  
\begin{align} \label{cond-1}
\av{{\gm_0}'(x)} \equiv 1, \,\, 
\gm_0(0) = (0, 0), \,\, 
\gm_0(1)=(R,0), \,\, 
\vk_0(0)= \vk_0(1)=0.  
\end{align}

The purpose of this subsection is applying Theorem \ref{main-thm} to the problem \eqref{n-o} 
and proveing that the solution $\gm(x,t)$ converges to a stationary solution as $t \to \infty$. 

Regarding Assumption \ref{Asum-B}, it is easy to check that the flow \eqref{s-s} with the 
boundary condition \eqref{fixed-p}-\eqref{D-curvature} is the $L^2$-gradient flow for the 
functional $\mE_\lm$ (see Section \ref{appendix-b}). 

By the proof of the following Proposition, we see that Assumption \ref{U-B} holds. 
\begin{prop}{\rm (\cite{novaga-okabe}) } \label{n-o-prop}
Let $\gm_0(x)$ be a planar curve satisfying \eqref{cond-1}. 
Then there exist a family of smooth planar curves 
$\gm(x,t) : I \times [0, \infty) \to \R^2$ satisfying \eqref{n-o}. 
Moreover, there exist sequence $\{ t_j \}^\infty_{j=1}$ 
and a smooth curve $\tilde{\gm} : I \to \R^2$ such that 
$\gm(\cdot,t_j)$ converges to $\tilde{\gm}(\cdot)$ as $t_j \to \infty$ up to 
a reparametrization. Moreover the curve $\tilde{\gm}$ satisfies  
\begin{align} \label{ep}
\begin{cases}
& 2 \pd^2_{s} \tilde{\vk} + \tilde{\vk}^3 - \lm^2 \tilde{\vk}=0 \quad \text{in} \quad I, \\ 
& \tilde{\gm}(0) = (0, 0), \,\, \tilde{\gm}(1)=(R,0), \,\, 
\tilde{\vk}(0)= \tilde{\vk}(1)=0. 
\end{cases}
\end{align}
\end{prop}

Let $\mS$ be a set of all stationary solutions, i.e., a set of all planar open curves satisfying \eqref{ep}. 
And for each $A \in \R$, let us define the set $\wSigma_A$ of $\mS$ as follows:  
\begin{align*} 
\wSigma_A = \{ \tilde{\gm} \in \mS \mid \mE_\lm(\tilde{\gm}) \leq A \}. 
\end{align*}
By making use of Lemma \ref{l-p-E-div},  we prove that the set $\wSigma_A$ is finite for any $A \in \R$: 
\begin{lem} \label{n-o-finite}
The set $\wSigma_A$ is finite for any $A \in \R$. 
\end{lem}
\begin{proof}
To begin with, we identify $\tilde{\gm} \in \mS$ with $\mathcal{R} \tilde{\gm} \in \mS$, 
where  $\mathcal{R}= \left( \begin{smallmatrix} 1 & 0 \\ 0 & -1 \end{smallmatrix} \right)$. 
If $\wSigma_A$ is not finite, there exists a sequence $\{ \tilde{\gm}_n \}^\infty_{n=1} \subset \wSigma_A$. 
Then there exists a constant $E_n \geq 0$ such that $(\tilde{\vk}_n, E_n)$ satisfies \eqref{f-integ} for each $n \in \N$. 
If $E_n \to \infty$ as $n \to \infty$, then Lemma \ref{l-p-E-div} implies that 
\begin{align*}
\int_{\tilde{\gm}_n} \tilde{\vk}^2_n \, ds \to \infty. 
\end{align*}
This contradicts $\{ \tilde{\gm}_n \}^\infty_{n=1} \subset \wSigma_A$. 
Thus there exists a constant $E^*$ such that $E_n < E^*$ holds for any $n \in \N$. 
Moreover, if $E_n \to 0$ as $n \to \infty$, then Lemma \ref{anal-L} implies that $L(E_n) \to \infty$ as $n \to \infty$. 
Then we observe that $\mE_\lm(\tilde{\gm}_n) \to \infty$ as $n \to \infty$. 
This also contradicts $\{ \tilde{\gm}_n \}^\infty_{n=1} \subset \wSigma_A$. 
Hence there exists a positive constant $E_*>0$ such that $E_* < E_n$ for all $n \in \N$. 
Since $\{ E_n \}^\infty_{n=1}$ is a bounded sequence, there exist a constant $E_* \leq E_\infty \leq E^*$ and 
a subsequence $\{ E_{n_k} \}^\infty_{k=1} \subset \{ E_n \}^\infty_{n=1}$ such that 
$E_{n_k} \to E_\infty$ as $k \to \infty$. 
By the definition of $\{ \tilde{\gm}_n \}^\infty_{n=1}$ and $\{ E_n \}^\infty_{n=1}$, 
there exists $\tilde{\gm}_\infty \in \wSigma_A$ such that $(\vk_\infty, E_\infty)$ satisfies \eqref{f-integ}. 
Here we define a function $d= d(E)$ for a planar open curves with the pair $(\vk, E)$ 
satisfying \eqref{f-integ} as 
\begin{align*}
d(E)= \av{\gm(L(E)) - \gm(0)}. 
\end{align*}
In Lemma \ref{anal-L}, we shall prove that $L(E)$ is analytic on $(0, \infty)$. 
Since $\gm(s)$ depends on $s$ analytically, the analyticity of $L(E)$ implies that $d(E)$ is analytic. 
Since $\tilde{\gm}_n \in \mS$, there exists a number $N_n \in \N$ such that 
\begin{align*}
d(E_n) = \dfrac{R}{N_n}. 
\end{align*} 
In particular, there exists a number $N_\infty \in \N$ such that 
\begin{align*}
d(E_\infty) = \dfrac{R}{N_\infty}. 
\end{align*}
Since $d(E_{n_k}) \to d(E_\infty)$ as $ k \to \infty$, 
it must be holds that $N_{n_k}= N_\infty$ for sufficiently large $k \in \N$. 
This means that $d(E_{n_k})= d(E_\infty)$ holds for sufficiently large $k \in \N$. 
The analyticity of $d(\cdot)$ implies that $E_{n_k}= E_\infty$ for sufficiently large $k \in \N$. 
The relation $E_{n_k}= E_\infty$ yields $\tilde{\gm}_{n_k} = \tilde{\gm}_\infty$.  
Since  $\tilde{\gm}_n \in \wSigma_A$ is uniquely determined with respect to $E_n$ by identifying 
$\tilde{\gm}_n$ with $\mathcal{R} \tilde{\gm}_n$, this contradicts the uniqueness. 
\end{proof}

Since Lemma \ref{n-o-finite} implies that Assumption \ref{A-B} holds, 
we see that Theorem \ref{main-thm} yields the following:   
\begin{thm} \label{appl-0}
Let $\gm(x,t) : I \times [0,\infty) \to \R^2$ be a solution of \eqref{n-o}. 
Then there exists a solution $\tilde{\gm}$ of \eqref{ep} such that 
\begin{align*}
\gm(\cdot,t) \to \tilde{\gm}(\cdot) \quad \text{as} \quad t \to \infty
\end{align*}
in the $C^\infty$-topology. 
\end{thm}

\subsection{Symmetric Navier boundary condition} 

We now consider the following more general boundary condition for the curvature: 
\begin{align} \label{Navier}
\vk(0)= \vk(1)= \va,  
\end{align}
where $\va \in \R$ is a given constant. The boundary conditions \eqref{fixed-p}-\eqref{Navier} 
is sometimes called the {\it symmetric Navier boundary condition} (e.g., see \cite{BGN, dec-gru}). 
In Section \ref{appendix-b}, we will show that the flow \eqref{s-s} with the symmetric Navier boundary condition 
is the $L^2$-gradient flow of the functional  
\begin{align}
\mE_{\lm, \va}(\gm):= \mE_\lm(\gm) - 2 \va \int_\gm \vk \, ds.  
\end{align}

We now show that the functional $\mE_{\lm, \va}$ is bounded from below whenever $|\alpha| < |\lambda|$. 
\begin{lem} \label{l-E-bd-b}
Let $\va,\lm\in \R$ be such that 
$\lm \ne 0$ and
\begin{equation}\label{vala}
\av{\va} < \av{\lm}.
\end{equation}
Then there exists a positive constant $C=C(\va, \lm)$ such that  
\begin{align} \label{E-bd-b}
\mE_{\lm,\va}(\gm) \geq C \max{\{ \Lg{\vk}^2, \mL(\gm) \}} \qquad {\rm for\ all\ }\gm.
\end{align} 
\end{lem}
\begin{proof}
Using H\"older's and Young's inequalities, for all $\ve\in (0,1]$ we have 
\begin{align} \label{p-E-lm-va}
\mE_{\lm,\va}(\gm) &=  \int_\gm \vk^2 \, ds - 2 \va \int_\gm \vk \, ds + \lm^2 \mL(\gm) \\
 & \geq \int_\gm \vk^2 \, ds - 2\av{\va} \left\{\int_\gm \vk^2 \, ds \right\}^{\frac{1}{2}} 
        \left\{ \int_\gm \, ds \right\}^{\frac{1}{2}} + \lm^2 \mL(\gm) \notag \\
 & \geq (1-\ve) \int_\gm \vk^2 \, ds + \left( \lm^2 - \dfrac{\va^2}{\ve} \right) \mL(\gm). \notag
\end{align}
Taking $\va^2 / \lm^2 < \ve < 1$, we obtain \eqref{E-bd-b}. 
\end{proof}

The purpose of this subsection is to prove a convergence of a solution of the following initial boundary 
value problem 
\begin{align} \label{g-curvature-flow}
\begin{cases}
& \pd_t \gm= -2 \pd^2_s \vk - \vk^3 + \lm^2 \vk \qquad \qquad \quad
\text{in} \quad I \times [0, \infty), \\
& \gm(0,t)= (0,0),\,\, \gm(1,t)= (R, 0),\\ 
& \vk(0,t)= \vk(1,t)= \va \qquad \qquad \qquad \quad 
\text{in} \quad [0, \infty), \\
& \gm(x,0)= \gm_0(x) \qquad \qquad \qquad \qquad \quad \,
\text{in} \quad I, 
\end{cases}
\end{align}
to a solution of 
\begin{align} \label{g-curvature}
\begin{cases}
& 2 \pd^2_s \vk + \vk^3 - \lm^2 \vk =0 \quad \text{in} \quad I, \\
& \gm(0)= (0,0),\,\, \gm(1)= (R, 0),\,\, \vk(0)= \vk(1)= \va, 
\end{cases}
\end{align}
as $t \to \infty$. 

In Section \ref{appendix-b}, we shall prove that there exists a unique smooth solution for all times,
satisfying Assumption \ref{U-B}. 

We turn to Assumption \ref{A-B}. 
Let $\gm$ be a planar open curve satisfying the stationary equation 
\begin{align} \label{g-curvature-sp}
\begin{cases}
& 2 \pd^2_s \vk + \vk^3 - \lm^2 \vk =0 \quad \text{in} \quad I, \\
& \vk(0)= \vk(1)= \va.     
\end{cases}
\end{align}
Then there exists a constant $E \in (-\lm^4/4, \infty)$ such that the pair $(\vk, E)$ satisfies 
\begin{align} \label{g-curvature-E}
\begin{cases}
& \left( \dfrac{d \vk}{ds} \right)^2 + F(\vk) = E \quad \text{in} \quad I, \\
& \vk(0)= \vk(1)= \va.     
\end{cases}
\end{align}
Let 
\begin{align}
L_0(E) &=0
\\
L_1(E) &= 2 \int^{\va}_{\vk_m} \dfrac{d\vk}{\sqrt{E-F(\vk)}}, \label{def-L1} 
\\
L_2(E) &= 2 \int^{\vk_M}_{\va} \dfrac{d\vk}{\sqrt{E-F(\vk)}}, \label{def-L2} 
\end{align}
so that 
\begin{align} \label{Le-Li}
L(E) = L_1(E) + L_2(E). 
\end{align}
It is easy to see that the length of $\gm$ can be written as
\begin{align*}
\mL(\gm)= \tilde{L}(E) + N L(E)
\end{align*} 
for some $N \in \N$, with 
\begin{align*}
\tilde{L}(E) \in \{ L_0(E), L_1(E), L_2(E) \} \quad \text{for any} \quad E \in ( F(\va), +\infty).
\end{align*}

Let $\mS$ be a set of all solutions of \eqref{g-curvature}. 
For each $A \in \R$, we define 
\begin{align} \label{S-A-Navi}
\wSigma_A= \{ \gm \in \mS \mid \mE_{\lm, \va}(\gm) \leq A \}. 
\end{align}
\begin{lem} \label{finite-p-curve}
Let $\va$, $\lm \in \R$ be such that $\lm \ne 0$ and \eqref{vala}. 
Then the set $\wSigma_A$ is finite for each $A \in \R$. 
\end{lem}

\begin{proof}
Assume by contradiction that there exists a sequence $\{ \gm_n \}^\infty_{n=1} \subset \wSigma_A$ 
with $\gm_l \neq \gm_m$ if $l \neq m$. 
Then there exists a constant $E_n$ for each $n \in \N$ such that the pair $(\vk_n, E_n)$ satisfies 
\begin{align} \label{vk_n}
\begin{cases}
\left( \dfrac{d \vk_n}{ds} \right)^2 + F(\vk_n) = E_n, \\
\vk_n(0)= \vk_n(\mL(\gm_n))= \va,  
\end{cases}
\end{align}
where $\vk_n$ denotes the curvature of $\gm_n$. 
By the discussion above, for each $E_n \in (-\lm^4/4, \infty)$ there exists a unique solution $\vk_n$ of \eqref{vk_n} 
such that $\mL(\gm_n)= L_i(E_n) + N_n L(E_n)$, where $i \in \{0, 1, 2 \}$. 

We claim that there exists a positive number $E^*$ such that $E_n \leq E^*$ for any $n \in \N$. 
Suppose that $E_n \to \infty$ as $n \to \infty$. 
Then Lemma \ref{l-p-E-div} yields
\begin{align*} 
\int_{\gm_n} \vk_n^2 \, ds \to \infty \quad \text{as} \quad n \to \infty.
\end{align*}
By virtue of Lemma \ref{l-E-bd-b}, this implies that $\mE_{\lm, \va}(\gm_n) \to \infty$ 
as $n \to \infty$,
which contradicts $\gm_n  \subset \wSigma_A$. Thus we see that $\{ E_n \}$ is a bounded sequence, 
and then, there exists 
a constant $E_\infty \leq E^*$ such that $E_{n} \to E_\infty$ up to extracting a suitable subsequence. 
Moreover, possibly passing to a further subsequence, there exists 
a  curve $\gm_\infty \in \wSigma_A$ such that the curves $\gm_n$ smoothly converge to $\gm_\infty$ as $n \to \infty$.
As $\mL(\gm_\infty)= L_i(E_\infty) + N_\infty L(E_\infty)$ for some $i \in \{ 0,1,2 \}$,  
it follows that $\mL(\gm_n)= L_i(E_n) + N_\infty L(E_n)$ for sufficiently large $n$. 
We define 
\begin{align*}
{\bm d}(E) &= \gm(L(E)) - \gm(0), \\
\tilde{\bm d}(E) &= \gm(L_i(E)) - \gm(0), 
\end{align*}
where $\gm$ is a solution of $(\vk_s)^2 + F(\vk) =E$. 
Since $\gm_n \in \wSigma_A$ and $N_n= N_\infty$ for $n$ big enough,
we have 
\begin{align} \label{equi-E-infinity}
| \tilde{\bm d}(E_n) + N_n {\bm d}(E_n) | =| \tilde{\bm d}(E_n) + N_\infty {\bm d}(E_n) | = R 
\end{align}
for $n$ sufficiently large.

In the following we show that \eqref{equi-E-infinity} leads to a contradiction. 
We may assume that $\va > 0$ without loss of generality. 
First we consider the case where $F(\va) \geq 0$. 
Since $F(\va) \geq 0$ implies $E_n >0$ for any $n \in \N$, Lemmas \ref{anal-L2}-\ref{anal-L3} imply that 
the function $|\tilde{\bm d}(E) + N_\infty {\bm d}(E)|$ is analytic on $(F(\va), \infty)$. 
Then \eqref{equi-E-infinity} yields 
\begin{align} \label{const-E>0}
| \tilde{\bm d}(E) + N_\infty {\bm d}(E) | =R \,\,\,\,\, \text{for any} \,\,\,\,\, E \in (F(\va), \infty). 
\end{align}
It follows from  Lemma \ref{anal-L3} that $L(E) \to 0$ as $E \to \infty$. 
Then \eqref{Le-Li} yields that $L_i(E) \to 0$ as $E \to \infty$ for any $i \in \{1,2 \}$. 
Thus we observe that 
\begin{align*}
| \tilde{\bm d}(E) + N_\infty {\bm d}(E) | 
 \leq | \tilde{\bm d}(E) | + N_\infty \av{{\bm d}(E)} \to 0 \quad \text{as} \quad E \to \infty. 
\end{align*}
This contradicts \eqref{const-E>0}. 

Next we consider the case where $F(\va)<0$ and $E_\infty \geq 0$. 
Since we may assume that $E_n \geq 0$ for sufficiently large $n \in \N$, 
we can obtain a contradiction along the same argument of the case where $F(\va) \geq 0$. 

Finally we consider the case where $F(\va)<0$ and $E_\infty < 0$. 
Since it holds that $E_n <0$ for sufficiently large $n$, 
Lemmas \ref{anal-L}--\ref{anal-L3} and \eqref{equi-E-infinity} yield that 
\begin{align} \label{const-E<0}
| \tilde{\bm d}(E) + N_\infty {\bm d}(E) | =R \quad \text{for any} \quad E \in (-\lm^4/4, 0). 
\end{align}
Suppose that $N_\infty \neq 0$. 
Since $N_\infty \geq 1$, we have 
\begin{align*}
| \tilde{\bm d}(E) + N_\infty {\bm d}(E) | 
 \geq N_\infty \av{{\bm d}(E)} - | \tilde{\bm d}(E) | 
 \geq N_\infty \left( \av{{\bm d}(E)} - | \tilde{\bm d}(E) | \right).  
\end{align*}
Remark that Lemmas \ref{anal-L}--\ref{anal-L3} and \eqref{Le-Li} imply that 
$L(E) \to \infty$, $L_1(E) \to \infty$, and $L(E) - L_2(E) \to \infty$ as $E \uparrow 0$. 
If $\tilde{L}(E) \in \{ L_0(E), L_1(E) \}$, then it holds that 
\begin{align*}
\av{{\bm d}(E)} - | \tilde{\bm d}(E) | \to \infty \quad \text{as} \quad E \uparrow 0, 
\end{align*}
and then 
 \begin{align*}
| \tilde{\bm d}(E) + N_\infty {\bm d}(E) | \to  \infty \quad \text{as} \quad E \uparrow 0. 
\end{align*}
This contradicts \eqref{const-E>0}. 
If $\tilde{L}(E)= L_2(E)$, since \eqref{Le-Li} gives us that 
\begin{align*}
\dfrac{L(E)}{L_1(E)} = 1 - \dfrac{L_2}{L_1} \to 1 \quad \text{as} \quad E \uparrow 0, 
\end{align*}
we observe that 
\begin{align} \label{d-td=0}
| d(E) - \tilde{d}(E) | \to 0 \quad \text{as} \quad E \uparrow 0. 
\end{align}
Then it follows from \eqref{d-td=0} that 
\begin{align*}
| \tilde{d}(E) + N_\infty d(E) | 
 \geq (N_\infty +1) |d(E)| - |\tilde{d}(E)-d(E)| \to \infty \quad \text{as} \quad E \uparrow 0. 
\end{align*}
This also contradicts \eqref{const-E>0}.
Thus it must hold that $N_\infty=0$. 
Then \eqref{const-E<0} is reduced to  
\begin{align} \label{td-const-E<0}
| \tilde{\bm d}(E) | =R \quad \text{for any} \quad E \in (-\lm^4/4, 0). 
\end{align} 
Lemma \ref{anal-L2} implies that $\tilde{L}(E) = L_2(E)$. 
With the aid of Lemma \ref{anal-L3}, we can replace \eqref{td-const-E<0} with 
\begin{align} \label{td-const-E<0-2}
| \tilde{\bm d}(E) | =R \quad \text{for any} \quad E \in (F(\va), \infty). 
\end{align} 
Moreover, by virtue of Lemma \ref{anal-L3}, we see that 
$L_2(E) \to 0$ as $E \to \infty$, i.e., $|\tilde{\bm d}(E)| \to 0$ as $E \to \infty$. 
This contradicts \eqref{td-const-E<0-2}. 
The proof of Lemma \ref{finite-p-curve} is complete. 
\end{proof}

Applying Theorem \ref{main-thm} to the problem \eqref{g-curvature-flow}, we obtain the following: 
\begin{thm} \label{appl-g}
Let $\va$, $\lm \in \R$ satisfy $\lm \neq 0$ and \eqref{vala}. 
Let $\gm(x,t) : I \times [0,\infty) \to \R^2$ be a solution of \eqref{g-curvature-flow}.
Then there exists a solution $\tilde{\gm}$ of \eqref{g-curvature}
\begin{align*}
\gm(\cdot,t) \to \tilde{\gm}(\cdot) \quad \text{as} \quad t \to \infty
\end{align*}
in the $C^\infty$-topology.  
\end{thm}


\section{Appendix A} \label{appendix}
\begin{lem} \label{anal-L}
Let $\gm(s) : [0, \infty) \to \R^2$ be a planar open curve with the curvature satisfying 
\begin{align*} 
\left( \dfrac{d\vk}{ds} \right)^2 + F(\vk)= E. 
\end{align*} 
Then the function 
\begin{align*} 
L(E)= 2 \int^{\vk_M(E)}_{\vk_m(E)} \dfrac{d \vk}{\sqrt{E - F(\vk)}}. 
\end{align*}
is analytic on $(-\lm^4/4, 0) \cup (0, \infty)$. Furthermore it holds that 
\begin{align} \label{LE-div}
L(E) \to \infty \quad \text{as} \quad E \to 0. 
\end{align}
\end{lem}
\begin{proof}
To begin with, we show that $L(E)$ is analytic on $(0,\infty)$ and $L(E) \to \infty$ as $E \downarrow 0$. 
Recall that $L(E)$ is written as 
\begin{align*}
L(E)= 4 \int^{\vk_M(E)}_0 \dfrac{d \vk}{\sqrt{E-F(\vk)}} 
\end{align*}
for $E \in (0,\infty)$. 
Since $F$ is analytic, it is clear that $\vk_M(E)$ is analytic. 
Moreover the definition of $\vk_M(E)$ implies that $F'(\vk(E)) \neq 0$. 
The Taylor expansion of $F$ at $\vk= \vk_M(E)$ is expressed as 
\begin{align*}
F(\vk)&= F(\vk_M) + F'(\vk_M) (\vk- \vk_M) + \dfrac{F''(\vk_M)}{2!} (\vk- \vk_M)^2 \\
          & \qquad + \dfrac{F^{(3)}(\vk_M)}{3!} (\vk- \vk_M)^3 + \dfrac{F^{(4)}(\vk_M)}{4!} (\vk- \vk_M)^4. 
\end{align*}
It follows from $F'(\vk_M) \neq 0$ that 
\begin{align*}
\sqrt{E-F(\vk)} = \sqrt{F'(\vk_M) (\vk_M - \vk)} 
                           \sqrt{1+ \sum^3_{n=1} a_n(E) (\vk_M - \vk)^n}  
\end{align*}
for any $\vk \in [0, \vk_M]$, where $a_n(E)$ is given by 
\begin{align*}
a_n(E)= \dfrac{(-1)^n F^{(n+1)}(\vk_M(E))}{(n+1)! F'(\vk_M(E))}
\end{align*} 
Since it holds that 
\begin{align*}
\av{a_n(E)} \leq C \av{\lm}^{-n}
\end{align*}
for any $E>0$, we see that  
\begin{align*}
\dfrac{1}{\sqrt{E-F(\vk)}} = \sum_{k=0}^\infty b_k (\vk_M(E) - \vk)^{k-1/2}
\end{align*}
for any $\vk \in (\vk_M(E)-\ve, \vk_M(E))$, where $\ve$ is a positive constant satisfying 
\begin{align*} 
\max_{1 \leq n \leq 3}{\av{a_n}} \ve (1 + \ve + \ve^2) < 1. 
\end{align*}
Remark that $b_k=b_k(E)$ is analytic on $(0,\infty)$. 
In the following let us set $L(E)/4= \mkL_1(E)+ \mkL_2(E)$, which are written as   
\begin{align*}
\mkL_1(E) = \int_0^{\vk_0(E)} \dfrac{d \vk}{\sqrt{E-F(\vk)}}, \quad 
\mkL_2(E) = \int_{\vk_0(E)}^{\vk_M(E)} \dfrac{d \vk}{\sqrt{E-F(\vk)}}, 
\end{align*}
where $\vk_0(E)= \vk_M(E) - \ve/2$.
First we check that $\mkL_1(E)$ is analytic. Let us write $\mkL_1(E)$ as 
\begin{align*}
\mkL_1(E)= \dfrac{1}{\sqrt{E}} \int^{\vk_0}_0 \left( 1 - \dfrac{F(\vk)}{E} \right)^{-\frac{1}{2}} d \vk.  
\end{align*}
Notice that the function $(1- y)^{-1/2}$ is analytic on $(-\infty, 1)$, and for any $y_0 < 1$ one can write 
\begin{align*}
(1-y)^{-\frac{1}{2}} = \sum^{\infty}_{k=0} c_k (y-y_0)  \quad \text{for all} \quad y \in (y_0, 1), 
\end{align*}
where the coefficients $c_k$ depend on $y_0$. 
Letting $\bar{\vk}= \av{\lm}$ which is a minimum point of $F$ and setting $y= F(\vk)/E$, $y_0= F(\bar{\vk})/E$, we have 
\begin{align} \label{expan-L1}
\mkL_1(E) &= \dfrac{1}{\sqrt{E}} \sum^\infty_{k=0} c_k 
                                          \int^{\vk_0}_0 \left( \dfrac{F(\vk)}{E} - \dfrac{F(\bar{\vk})}{E} \right)^k d \vk \\
                 &= \sum^\infty_{k=0} c_k \dfrac{g_k(\vk_0(E))}{E^{k+\frac{1}{2}}}, \notag
\end{align}
where 
\begin{align*}
g_k(x)= \int^x_0 (F(\vk) - F(\bar{\vk}) )^k d \vk. 
\end{align*}
Since it holds that 
\begin{align*}
\left( \int^{\vk_0}_0 \left( \dfrac{F(\vk)}{E} - \dfrac{F(\bar{\vk})}{E} \right)^k d \vk \right)^{\frac{1}{k}} 
 \to \su{\vk \in (0, \vk_0)}{\dfrac{F(\vk)- F(\bar{\vk})}{E}}
 < 1-y_0 
\end{align*}
as $k \to \infty$, we see that the series in \eqref{expan-L1} converges for each $E>0$. 
Recalling $\vk_M(E)$ is analytic, all the functions $g_k(\vk_0(E))$ is also analytic. 
This implies that $\mkL_1(E)$ is analytic for $E>0$. 

Regarding $\mkL_2(E)$, we have 
\begin{align*}
\mkL_2(E) = \int_{\vk_0}^{\vk_M(E)} \dfrac{d\vk}{\sqrt{E-F(\vk)}}
       = -\sum_{k=0}^\infty \dfrac{b_{k}(E)}{(k+1/2)} \left( \dfrac{\ve}{2} \right)^{k+1/2}. 
\end{align*}
Since $b_k(E)$ is analytic, this implies that $\mkL_2(E)$ is also analytic for $E>0$. 
Therefore we observe that $L(E)$ is analytic on $(0, \infty)$. 
On the other hand, it follows from \eqref{expan-L1} that 
\begin{align} \label{Le-div-above}
L(E) \to \infty \quad \text{as} \quad E \downarrow 0. 
\end{align}

Next we prove that the function $L(E)$ is analytic on $(-\lm^4/4, 0)$. 
Along the same line as above, we see that 
\begin{align*}
\dfrac{1}{\sqrt{E-F(\vk)}} = \sum_{k=0}^\infty b_k (\vk_M(E) - \vk)^{k-1/2}
\end{align*}
for any $\vk \in (\vk_M-\ve, \vk_M)$, and 
\begin{align*}
\dfrac{1}{\sqrt{E-F(\vk)}} = \sum_{k=0}^\infty \tilde{b}_k (\vk_m(E) - \vk)^{k-1/2}
\end{align*}
for any $\vk \in (\vk_m, \vk_m + \ve)$, where $\ve$ is an appropriate small number. 
Setting $L(E)= \tilde{\mkL}_1(E) + \tilde{\mkL}_2(E) + \tilde{\mkL}_3(E) + \tilde{\mkL}_4(E)$, where 
\begin{align*}
\tilde{\mkL}_1(E) &= \int^{\vk_m + \ve/2}_{\vk_m} \dfrac{d \vk}{\sqrt{E-F(\vk)}}, \qquad
\tilde{\mkL}_2(E)  = \int^{\bar{\vk}}_{\vk_m + \ve/2} \dfrac{d \vk}{\sqrt{E-F(\vk)}}, \\
\tilde{\mkL}_3(E) &= \int^{\vk_M - \ve/2}_{\bar{\vk}} \dfrac{d \vk}{\sqrt{E-F(\vk)}}, \qquad
\tilde{\mkL}_4(E)  = \int^{\vk_M}_{\vk_M - \ve/2} \dfrac{d \vk}{\sqrt{E-F(\vk)}}. 
\end{align*}
Regarding $\tilde{\mkL}_1(E)$ and $\tilde{\mkL}_4(E)$, we can verify that $\tilde{\mkL}_1(E)$ and $\tilde{\mkL}_4(E)$ are 
analytic on $(-\lm^4/4, 0)$ along the same argument for $\mkL_2(E)$. 
Next we turn to $\tilde{\mkL}_2(E)$. 
Along the same line as the argument for $\mkL_1(E)$, we have 
\begin{align} \label{tL2}
\tilde{\mkL}_2(E) 
 &= \int^{\bar{\vk}}_{\vk_m + \ve/2} \dfrac{1}{\sqrt{-F(\vk)}} \dfrac{d \vk}{\sqrt{1- E/F(\vk)}} \\
 &= \sum^\infty_{k=0} \tilde{c}_k \int^{\bar{\vk}}_{\vk_m + \ve/2} \dfrac{1}{\sqrt{-F(\vk)}} 
                                                   \left( \dfrac{E}{F(\vk)} - \dfrac{E}{F(\bar{\vk})} \right)^k \,d \vk \notag \\
  &= \sum^\infty_{k=0} \tilde{c}_k \tilde{g}_k(\vk_m(E) + \ve/2) E^k, \notag
\end{align}
where 
\begin{align*}
\tilde{g}_k(x)= \int^{\bar{\vk}}_x \left( F(\vk)^{-1} - F(\bar{\vk})^{-1} \right) (-F(\vk))^{-1/2} \, d\vk. 
\end{align*}
Since it holds that
\begin{align*}
\tilde{\mkL}_2(E) 
&\leq \dfrac{1}{\sqrt{-E}} \sum^\infty_{k=0} \tilde{c}_k \int^{\bar{\vk}}_{\vk_m + \ve/2} 
                                                   \left( \dfrac{E}{F(\vk)} - \dfrac{E}{F(\bar{\vk})} \right)^k \,d \vk
\end{align*}
and 
\begin{align*}
& \left( \int^{\bar{\vk}}_{\vk_m + \ve/2}  
        \left( \dfrac{E}{F(\vk)} - \dfrac{E}{F(\bar{\vk})} \right)^k \,d \vk \right)^{\frac{1}{k}} \\
& \qquad \qquad \qquad 
\to \su{\vk \in (\vk_m + \ve/2, \bar{\vk})}{\dfrac{E}{F(\vk)} - \dfrac{E}{F(\bar{\vk})} }
  < 1-y_0 
\end{align*}
as $k \to \infty$, we observe that the series in \eqref{tL2} converges for each $E \in (-\lm^4/4,0)$. 
Recalling $\vk_m(E)$ is analytic, all the functions $\tilde{g}_k(\vk_m(E) + \ve/2)$ is also analytic. 
This implies that $\tilde{\mkL}_2(E)$ is analytic for $E  \in (-\lm^4/4,0)$. 
Since similar argument gives us that $\tilde{\mkL}_3(E)$ is also analytic for $E  \in (-\lm^4/4,0)$. 

Finally we prove that $L(E) \to \infty$ as $E \uparrow 0$. 
Regarding $\tilde{\mkL}_1(E)$, it holds that 
\begin{align} \label{mkL-1-div}
\tilde{\mkL}_1(E) & > \int^{\vk_m + \ve/2}_{\vk_m} \dfrac{d\vk}{\sqrt{-F'(\vk_m)(\vk-\vk_m) - F''(\vk_m)(\vk-\vk_m)^2}} \\
 &= \dfrac{1}{2 \sqrt{- F''(\vk_m)}} \log{\dfrac{1 + \sqrt{\frac{\ve}{\ve + 2a}}}{1- \sqrt{\frac{\ve}{\ve+2a}}}},  \notag
\end{align}
where $a= F'(\vk_m)/ F''(\vk_m)$. Since $F'(\vk_m(E)) \to 0$ and $F''(\vk_m(E)) \to -\lm^2$ as $E \uparrow 0$, 
it follows from \eqref{mkL-1-div} that 
\begin{align*}
\tilde{\mkL}_1(E) \to \infty \quad \text{as} \quad E \uparrow 0. 
\end{align*}
This clearly implies that $L(E) \to \infty$ as $E \uparrow 0$. 
\end{proof}
The arguments in the proof of Lemma \ref{anal-L} also implies an analyticity of $L_i(E)$ 
which are defined by \eqref{def-L1}--\eqref{def-L2}. 
\begin{lem} \label{anal-L2}
Let $\va>0$. If $F(\va) \geq 0$, then the function $L_1(E)$ is analytic on $(F(\va), \infty)$. 
If $F(\va) < 0$, then the function $L_1(E)$ is analytic on $(F(\va), 0) \cup (0, \infty)$.  
Moreover, as $E \to 0$, it holds that  
\begin{align} \label{L2-div}
L_1(E) \to \infty \quad \text{as} \quad E \to 0.  
\end{align}
\end{lem}
\begin{proof}
The proof of Lemma \ref{anal-L} gives us the conclusion. 
\end{proof}

\begin{lem} \label{anal-L3}
For each $\va >0$, the function $L_2(E)$ is analytic on $(F(\va), \infty)$. 
Moreover, for each $\va \in (0, 2 \sqrt{\av{\lm}})$, it holds that  
\begin{align} \label{L3-div}
L_2(E) \to 0 \quad \text{as} \quad E \to \infty.  
\end{align}
\end{lem}
\begin{proof}
An analyticity of $L_2(E)$ is followed from the same argument of the proof of Lemma \ref{anal-L}. 
We shall prove \eqref{L3-div}. 
Since $0< \va < \sqrt{2} \av{\lm}$, we divide $L_2(E)$ into two part as follows: 
\begin{align} \label{part-L3-1}
\int^{\vk_M}_{\va} \dfrac{d \vk}{\sqrt{E-F(\vk)}} 
 = \int^{\sqrt{2} \av{\lm}}_{\va} \dfrac{d \vk}{\sqrt{E-F(\vk)}} + \int^{\vk_M}_{\sqrt{2} \av{\lm}} \dfrac{d \vk}{\sqrt{E-F(\vk)}}. 
\end{align}
Recalling $F(\sqrt{2} \av{\lm})=0$, we have 
\begin{align*}
\int^{\sqrt{2} \av{\lm}}_{\va} \dfrac{d \vk}{\sqrt{E-F(\vk)}} 
\leq \dfrac{1}{\sqrt{E}} \int^{\sqrt{2} \av{\lm}}_{\va} d \vk 
\to 0 \quad \text{as} \quad E \to \infty. 
\end{align*}
Thus it is sufficient to estimate the second term of the right-hand side of \eqref{part-L3-1}. 
By changing the variable $\vk/(4E)^{1/4}=x$, we have 
\begin{align*}
\int^{\vk_M}_{\sqrt{2} \av{\lm}} \dfrac{d \vk}{\sqrt{E-F(\vk)}} 
 &\leq \dfrac{1}{\sqrt{E}} \int^{\vk_M}_{\sqrt{2} \av{\lm}} \dfrac{d \vk}{\sqrt{1 - \frac{\vk^4}{4E}}} \\
 &= \dfrac{\sqrt{2}}{E^{1/4}} \int^{\vk_M/(4E)^{1/4}}_{\sqrt{2} \av{\lm}/(4E)^{1/4}} \dfrac{dx}{\sqrt{1 - x^4}}. 
\end{align*}
And then, the conclusion is obtained from the following calculation: 
\begin{align*}
\dfrac{\sqrt{2}}{E^{1/4}} \int^{\vk_M/(4E)^{1/4}}_{\sqrt{2} \av{\lm}/(4E)^{1/4}} \dfrac{dx}{\sqrt{1 - x^4}}
 &\leq \dfrac{\sqrt{2}}{E^{1/4}} \int^{\vk_M/(4E)^{1/4}}_{\sqrt{2} \av{\lm}/(4E)^{1/4}} \dfrac{dx}{\sqrt{1 - x^2}} \\
 &= \dfrac{\sqrt{2}}{E^{1/4}} \left\{ \sin^{-1}{\dfrac{\vk_M}{(4E)^{1/4}}} - \sin^{-1}{\dfrac{\sqrt{2} \av{\lm}}{(4E)^{1/4}}} \right\} \\
 & \to 0 \quad \text{as} \quad E \to \infty. 
\end{align*}
\end{proof}


\section{Appendix B} \label{appendix-b}
The scope of this appendix is to prove that \eqref{g-curvature-flow} has a unique smooth solution defined for all times. 

Let us first show that the $L^2$-gradient flow for the functional $\mE_{\lm, \va}$ under \eqref{fixed-p}-\eqref{Navier} 
can be written as \eqref{s-s}.  Indeed, let $\gm: [0,1] \to \R^2$ be a smooth planar curve 
satisfying the symmetric Navier boundary condition 
\begin{align} \label{Navier-3}
\gm(0)=(0,0), \,\,\, \gm(1)=(R,0),\,\,\, 
\vk(0)=\vk(1)=\va, 
\end{align}
We consider a variation of $\gm$ defined as follows: 
\begin{align*}
\gm(x,\ve)= \gm(x) + \phi(x,\ve) \bn(x),  
\end{align*}
where $\bn$ is the unit normal vector, pointing in the direction of the curvature, given by 
\begin{align*}
\bn= \begin{pmatrix}
0 & -1 \\
1 & 0
\end{pmatrix}
\dfrac{\gm_x}{\av{\gm_x}}
:= \mR \dfrac{\gm_x}{\av{\gm_x}}, 
\end{align*}
and $\phi(x,\ve) \in C^\infty((-\ve_0, \ve_0); C^\infty(0,1))$ is an arbitral smooth function with 
\begin{align*}
\phi(x,0) \equiv \phi(0,\ve) \equiv \phi(1,\ve) \equiv 0. 
\end{align*}
In the following we shall derive a first variational formula for the functional $\mE_{\lm, \va}(\gm)$. 
Put 
\begin{align*}
\tau= \dfrac{\gm_x}{\av{\gm_x}}.  
\end{align*}
Since the curvature of $\gm$ is expressed as 
\begin{align} \label{curv-form}
\vk= \dfrac{\gm_{xx} \cdot \mR \gm_x}{\av{\gm_x}^3}, 
\end{align}
we have  
\begin{align*}
\vk= \gm_{xx} \cdot \bn \av{\gm_x}^{-2}, 
\end{align*}
and then Frenet-Serret's formula $\pd_s \bn \cdot \tau = - \vk$ yields that 
\begin{align*}
\bn_x \cdot \tau = - \vk \av{\gm_x}. 
\end{align*}
To begin with, we derive useful variational formulae. 
First we find the first variational formula of the local length.  
\begin{align} \label{f-v-length}
\dfrac{d}{d\ve}\av{\gm_x(x,\ve)}\biggm|_{\ve=0}
 &= \dfrac{\gm_x \cdot (\phi_\ve \bn)_x}{\av{\gm_x}} = \tau \cdot \phi_{\ve} \bn_x
  = - \vk \av{\gm_x} \phi_\ve,  
\end{align}
where $\phi_\ve(\cdot)= (\pd \phi/ \pd \ve)(\cdot, 0)$. 
Next we find the first variation formula of the curvature. 
From \eqref{curv-form} and 
\begin{align*}
\bn \cdot \bn_{x x} &= -\av{\bn_x}^2= - \vk^2 \av{\gm_x}^2, \\
\gm_{xx} \cdot \mR \bn_x &= \gm_{xx} \cdot \mR(-\vk \gm_x) = -\vk^2 \av{\gm_x}^3, \\
\left( \av{\gm_x}^{-1} \right)_x &= \dfrac{\gm_{xx} \cdot \mR \bn}{\av{\gm_x}^2}, 
\end{align*}
it follows that 
\begin{align} \label{f-v-curv}
\dfrac{d}{d\ve}\vk(x,\ve)\biggm|_{\ve=0} 
= \dfrac{\phi_{\ve x x}}{\av{\gm_x}^2} + \vk^2 \phi_\ve 
    + ( \av{\gm_x}^{-1} )_x \dfrac{\phi_{\ve x}}{\av{\gm_x}}. 
\end{align}
Using \eqref{f-v-length}, we obtain 
\begin{align*}
& \dfrac{d}{d\ve} \mE_{\lm,\va}(\gm(\cdot,\ve)) \biggm|_{\ve=0} \\ 
& \quad 
 = \int^1_0 \left\{ 2(\vk - \va) \dfrac{d}{d\ve}\vk \biggm|_{\ve=0} 
         - \left( \vk^3 - 2\va \vk^2 + \lm^2 \vk \right)  \phi_\ve \right\}\av{\gm_x} \, dx
\end{align*}
Using \eqref{f-v-curv} and integrating by parts, we get   
\begin{align*}
& \int^1_0 (\vk-\va) \dfrac{d}{d\ve}\vk \biggm|_{\ve=0} \av{\gm_x} \, dx \\
& = \int^1_0 (\vk-\va) \left\{ \dfrac{\phi_{\ve x x}}{\av{\gm_x}^2} + \vk^2 \phi_\ve 
    + ( \av{\gm_x}^{-1} )_x \dfrac{\phi_{\ve x}}{\av{\gm_x}} \right\} \av{\gm_x} \, dx \\
&=  \int^1_0 -\left(\dfrac{\vk-\va}{\av{\gm_x}} \right)_x \phi_{\ve x} 
    + (\vk^3-\va \vk^2) \phi_\ve \av{\gm_x} 
    + ( \av{\gm_x}^{-1} )_x (\vk-\va) \phi_{\ve x} \, dx \\
& \qquad + \left[ \dfrac{\vk-\va}{\av{\gm_x}} \phi_{\ve x} \right]^1_0 \\
& = \int^1_0 -\dfrac{\vk_x}{\av{\gm_x}} \phi_{\ve x} + (\vk^3-\va \vk^2) \phi_\ve \av{\gm_x} \, dx \\
& = \int^1_0 \left(\dfrac{\vk_x}{\av{\gm_x}}\right)_x \phi_\ve 
      + (\vk^3 - \va \vk^2) \phi_\ve \av{\gm_x} \, dx \\
& = \int^1_0 \left\{ \left(\dfrac{\pd_x}{\av{\gm_x}}\right)^2 \vk 
         + (\vk^3 - \va \vk^2) \right\} \phi_\ve \av{\gm_x} \, dx. 
\end{align*}
Here we use $\vk(0)= \vk(1)= \va$. Thus we find 
\begin{align} \label{f-vari}
\dfrac{d}{d\ve} \mE_{\lm,\va} (\gm(\cdot,\ve)) \biggm|_{\ve=0} 
 = \int^b_a \left\{ 2 \left(\dfrac{\pd_x}{\av{\gm_x}}\right)^2 \vk 
         + \vk^3 - \lm^2 \vk \right\}\phi_\ve \av{\gm_x} \, dx. 
\end{align}
Parameterizing by the arc length, the formula \eqref{f-vari} is written as 
\begin{align*}
\dfrac{d}{d\ve} \mE_{\lm,\va} (\gm(\cdot,\ve)) \biggm|_{\ve=0} 
 = \int^1_0 \left\{ 2 \vk_{ss} + \vk^3 - \lm^2 \vk \right\}\phi_\ve \, ds. 
\end{align*}
Therefore we see that the flow \eqref{s-s} is the $L^2$-gradient flow 
for the functional $\mE_{\lm,\va}$ under the symmetric Navier boundary condition \eqref{Navier-3}. 

Since \eqref{g-curvature-flow} is a nonlinear boundary value problem for a quasi-linear parabolic equation, 
a short time existence is a standard matter. In what follows we shall prove a long time existence of 
solutions to \eqref{g-curvature-flow}. Throughout the section, put 
\begin{align*}
V^\lm = 2 \pd^2_s \vk + \vk^3 - \lm^2 \vk. 
\end{align*}
Then the equation in \eqref{g-curvature-flow} is written as 
\begin{align} \label{s-s-flow}
\pd_t \gm = - V^\lm \bn. 
\end{align}
Since $s$ depends on $t$, remark that the following holds.  
\begin{lem} \label{comu}
Under \eqref{s-s-flow}, the following commutation rule holds{\rm :} 
\begin{align*}
\pd_t \pd_s = \pd_s \pd_t - \vk V^\lm \pd_s. 
\end{align*}
\end{lem}
Lemma \ref{comu} gives us the following: 
\begin{lem} \label{k-flow}
Let $\gm(x,t)$ satisfy \eqref{s-s-flow}. Then it holds that 
\begin{align} \label{k-eq}
\pd_t \vk & = - \pd^2_s V^\lm - \vk^2 V^\lm. 
\end{align}
Furthermore, the line element $ds$ of $\gm(x,t)$ satisfies 
\begin{align} \label{ds-pd}
\pd_t ds = \vk V^\lm ds. 
\end{align}
\end{lem}
The boundary conditions in \eqref{g-curvature-flow} imply that several terms vanish on the boundary. 
\begin{lem}\label{vanishing}
Suppose that $\gm$ satisfies \eqref{g-curvature-flow}. Then  it holds that  
\begin{align}
\pd_t \gm &=0 \quad \quad \, \text{on} \quad \pd I \times [0,\infty), \label{v-gm-t} \\
\pd_t \vk &=0 \quad \quad \, \text{on} \quad \pd I \times [0,\infty), \label{v-k-t} \\
V^\lm &=0  \quad \quad \, \text{on} \quad \pd I \times [0,\infty), \label{v-V} \\
\pd^2_s V^\lm &=0  \quad \quad \, \text{on} \quad \pd I \times [0,\infty), \label{v-V-s2} \\
\pd_t V^\lm &=0  \quad \quad \, \text{on} \quad \pd I \times [0,\infty), \label{v-V-t} \\
\pd_t \pd^2_s V^\lm &=0  \quad \quad \, \text{on} \quad \pd I \times [0,\infty), \label{v-V-ts2} \\
\pd_t \pd_s &= \pd_s \pd_t  \quad \text{on} \quad \pd I \times [0,\infty) \label{can-commu}. 
\end{align}
\end{lem}
\begin{proof}
Since both $\gm(t)$ and $\vk(t)$ are fixed on $\pd I$, we observe \eqref{v-gm-t}-\eqref{v-k-t}. 
It follows from \eqref{s-s-flow} and \eqref{v-gm-t} that \eqref{v-V} holds. 
By virtue of \eqref{k-flow}, \eqref{v-k-t}, and \eqref{v-V}, we obtain \eqref{v-V-s2}. 
Then \eqref{v-V} and \eqref{v-V-s2} implies \eqref{v-V-t} and \eqref{v-V-ts2}, respectively. 
\eqref{can-commu} is followed from Lemma \ref{comu} and \eqref{v-V}
\end{proof}

Here we introduce interpolation inequalities for open curves, which has been inspired by \cite{dziuk} for closed curves 
and given in \cite{lin}. The interpolation inequalities are written in terms of the following the scale invariant Sobolev norms: 
\begin{align*}
\Lp{\vk}{k,p}:= \sum^{k}_{i=0} \Lp{\pd^i_s \vk}{p}, \quad 
\Lp{\pd^i_s \vk}{p} := \mL(\gm)^{i+1-1/p} \left( \int_I \av{\pd^i_s \vk}^p \right)^{1/p}. 
\end{align*}
\begin{lem}{\rm (\cite{lin})} 
Let $\gm : I \to \R^2$ be a smooth curve. 
Then for any $k \in \N \cup \{ 0 \}$, $p \geq 2$, and $0 \leq i < k$, we have 
\begin{align*}
\Lp{\pd^i_s \vk}{p} \leq c \Lp{\vk}{2}^{1-\va} \Lp{\vk}{k,2}^\va, 
\end{align*}
where $\va= (i+ \tfrac{1}{2} - \tfrac{1}{p})/k$ and $c= c(n,k,p)$. 
\end{lem}
In order to prove a long time existence of solutions to \eqref{g-curvature-flow}, we make use of the following Lemma, 
which is a modification of Lemma 2.2 in \cite{dziuk}. 
\begin{lem} \label{l-e-method}
Let $\gm : I \times [0,T) \to \R^2$ satisfy the equation \eqref{s-s-flow} 
and $\phi : I \times [0,T) \to \R$ be a scalar function defined on $\gm$ satisfying 
\begin{align} \label{eq-phi}
\begin{cases}
& \pd_t \phi = - 2\pd^4_s \phi + Y \quad \text{in} \quad I \times [0,T), \\
& \phi=0, \,\,\,\, \pd^2_s \phi=0 \quad \,\,\, \text{on} \quad \pd I \times [0, T). 
\end{cases}
\end{align}
Then it holds that 
\begin{align} \label{e-method}
\dfrac{d}{dt} \dfrac{1}{4} \int_\gm \phi^2 \, ds + \int_\gm (\pd^2_s \phi)^2 \, ds 
 = \dfrac{1}{2}\int_\gm \phi Y \, ds + \dfrac{1}{4} \int_\gm \phi^2 \vk V^\lm \, ds.  
\end{align}
\end{lem}
\begin{proof}
It follows from the equation in \eqref{eq-phi} and Lemma \ref{k-flow} that 
\begin{align*}
\dfrac{d}{dt} \dfrac{1}{4} \int_\gm \phi^2 \, ds 
 &= \dfrac{1}{2}\int_\gm \phi \pd_t \phi \, ds + \dfrac{1}{4} \int_\gm \phi^2 \pd_t(ds) \\
 &= \dfrac{1}{2}\int_\gm \phi (- 2\pd^4_s \phi + Y) \, ds + \dfrac{1}{4} \int_\gm \phi^2 \vk V^\lm ds. 
\end{align*} 
With the aid of the boundary conditions in \eqref{eq-phi}, we obtain 
\begin{align*}
\int_\gm \phi \pd^4_s \phi \, ds 
 = - \int_\gm \pd_s \phi \pd^3_s \phi \, ds 
 =   \int_\gm (\pd^2_s \phi)^2 \, ds. 
\end{align*}
Then we observe \eqref{e-method}. 
\end{proof}

By virtue of Lemma \ref{vanishing}, we observe that $\pd^m_t V^\lm=0$ and $\pd^2_s \pd^m_t V^\lm=0$ hold 
on $\pd I$ for any $m \in \N \cup \{ 0 \}$. 
The fact implies that we can apply Lemma \ref{l-e-method} to $\phi= \pd^m_t V^\lm$.  
To do so, first we introduce the following notation for a convenience. 
\begin{dfn} {\rm (\cite{b-m-n})}  \label{q-def}
We use the symbol $\mfq^r(\pd^l_s \vk)$ for a polynomial with constant coefficients 
such that each of its monomials is of the form 
\begin{align*}
\prod^{N}_{i=1} \pd^{j_{i}}_s \vk 
\quad \text{\rm with} \quad 0 \leq j_{i} \leq l 
\quad \text{\rm and} \quad N \geq 1
\end{align*}
with 
\begin{align*}
r= \sum^{N}_{i=1}(j_{i} +1). 
\end{align*}
\end{dfn}
Making use of the notation, we obtain the following: 
\begin{lem} \label{formula-1}
Suppose that $\gm : I \times [0, \infty) \to \R^2$ satisfies \eqref{g-curvature-flow}. 
Let $\phi$ be a scalar function defined on $\gm$. 
Then  the following formulae hold for any $m$, $l \in \N${\rm :} 
\begin{align}
\pd^m_s V^\lm &= \mfq^{3+m}(\pd^{2+m}_s \vk) - \lm^2 \pd^m_s \vk, \label{V-s-m} \\
\pd_t \pd^m_s \phi 
 &= \pd^m_s \pd_t \phi + \sum^{m-1}_{i=0} (\mfq^{4+i}(\pd^{2+i}_s \vk) + \mfq^{2+i}(\pd^i_s \vk)) \pd^{m-i}_s \phi, \label{commu-t-s-m} \\
\pd_t \pd^m_s \vk
 &= -2 \pd^{m+4}_s \vk + \mfq^{m+5}(\pd^{m+2}_s \vk) + \mfq^{m+3}(\pd^{m+2}_s \vk), \label{k-t-sm} \\ 
\pd_t \mfq^l(\pd^m_s \vk) 
 &= \mfq^{l+4}(\pd^{m+4}_s \vk) + \mfq^{l+2}(\pd^{m+2}_s \vk). \label{mfq-t}
\end{align}
\end{lem}
\begin{proof}
Since $V^\lm = \mfq^3(\pd^2_s \vk) - \lm^2 \vk$, the assertion \eqref{V-s-m} is followed from a simple calculation. 
Regarding \eqref{commu-t-s-m}, we proceed by induction on $m$. For $m=1$, we have 
\begin{align*}
\pd_t \pd_s \phi 
 = \pd_s \pd_t \phi - \vk V^\lm \pd_s \phi 
 = \pd_s \pd_t \phi - (\mfq^4(\pd^2_s \vk) + \mfq^2(\vk)) \pd_s \phi. 
\end{align*}
Assuming that \eqref{commu-t-s-m} is true for some $m \geq 1$, we obtain 
\begin{align*}
\pd_t \pd^{m+1}_s \phi 
 &= \pd_s \pd_t \pd^m_s \phi + (\mfq^4(\pd^2_s \vk) + \mfq^2(\vk)) \pd^{m+1}_s \phi \\
 &= \pd_s \left\{ \pd^m_s \pd_t \phi + \sum^{m-1}_{i=0} (\mfq^{4+i}(\pd^{2+i}_s \vk) 
                         + \mfq^{2+i}(\pd^i_s \vk)) \pd^{m-i}_s \phi \right\} \\
 & \qquad + (\mfq^4(\pd^2_s \vk) + \mfq^2(\vk)) \pd^{m+1}_s \phi \\  
 &=  \pd^{m+1}_s \pd_t \phi + \sum^{m}_{i=0} (\mfq^{4+i}(\pd^{2+i}_s \vk) + \mfq^{2+i}(\pd^i_s \vk)) \pd^{m+1-i}_s \phi.      
\end{align*}
\eqref{k-t-sm} is followed from \eqref{k-flow} and \eqref{commu-t-s-m} directly. 
Finally we obtain \eqref{mfq-t} for $m$, $l \in \N$ fixed arbitrarily as follows: 
\begin{align*}
\pd_t \mfq^l(\pd^m_s \vk) 
 &= \sum^m_{j=0} \mfq^{l-j-1}(\pd^m_s \vk) \cdot \pd_t \pd^j_s \vk \\
 &= \sum^m_{j=0} \mfq^{l-j-1}(\pd^m_s \vk) \cdot \{ -2 \pd^{j+4}_s \vk + \mfq^{j+5}(\pd^{j+2}_s \vk) + \mfq^{j+3}(\pd^{j+2}_s \vk) \} \\
 &= \sum^m_{j=0} \mfq^{l+4}(\pd^{\max{\{ m, j+4 \}}}_s \vk) 
       + \sum^m_{j=0} \mfq^{l+2}(\pd^{\max{\{ m, j+2 \}}}_s \vk) \\
 &= \sum^{4}_{j=0} \mfq^{l+4}(\pd^{m+j}_s \vk) + \sum^{2}_{j=0} \mfq^{l+2}(\pd^{m+j}_s \vk) \\
 &= \mfq^{l+4}(\pd^{m+4}_s \vk) + \mfq^{l+2}(\pd^{m+2}_s \vk). 
\end{align*}
\end{proof}
With the aid of Lemma \ref{formula-1}, we obtain a representation of $\pd^m_t V^\lm$. 
\begin{lem}
For each $m \in \N$, it holds that 
\begin{align} \label{V-t-m}
\pd^m_t V^\lm 
 &= (-1)^m 2^{m+1} \pd^{4m+2}_s \vk + \mfq^{4m+3}(\pd^{4m}_s \vk) \\
 & \qquad \qquad \quad + \sum^{m}_{j=1} \mfq^{4m+3-2j}(\pd^{4m+2-2j}_s \vk). \notag
\end{align}
\end{lem}
\begin{proof}
We proceed by induction on $m$. For $m=1$, we have  
\begin{align*}
\pd_t V^\lm 
 &= \pd_t (2 \pd^2_s \vk + \vk^3 - \lm^2 \vk) \\
 &= 2 ( -2 \pd^6_s \vk + \mfq^7(\pd^4_s \vk) + \mfq^5(\pd^4_s \vk) ) + 3 \vk^2 \pd_t \vk - \lm^2 \pd_t \vk \\
 &= -2^2 \pd^6_s \vk + \mfq^7(\pd^4_s \vk) + \mfq^5(\pd^4_s \vk).  
\end{align*}
Suppose that \eqref{V-t-m} holds for $m=k$. Then we have 
\begin{align} \label{calc-1}
 \pd^{k+1}_t V^\lm 
&= \pd_t \{ (-1)^k 2^{k+1} \pd^{4k+2}_s \vk + \mfq^{4k+3}(\pd^{4k}_s \vk) \\
& \qquad + \sum^{k}_{j=1} \mfq^{4k+3-2j}(\pd^{4k+2-2j}_s \vk) \} \notag \\
 &=  (-1)^k 2^{k+1} \{ -2 \pd^{4k+6}_s \vk + \mfq^{4k+7}(\pd^{4k+6}_s \vk) + \mfq^{4k+5}(\pd^{4k+6}_s \vk)  \} \notag \\
 & \qquad + \pd_t \{  \mfq^{4k+3}(\pd^{4k}_s \vk) + \sum^{k}_{j=1} \mfq^{4k+3-2j}(\pd^{4k+2-2j}_s \vk) \}. \notag
\end{align}
By virtue of \eqref{mfq-t}, the last term in \eqref{calc-1} is reduced to 
\begin{align*} 
\pd_t & \{  \mfq^{4k+3}(\pd^{4k}_s \vk) + \sum^{k}_{j=1} \mfq^{4k+3-2j}(\pd^{4k+2-2j}_s \vk) \} \\
 &= \mfq^{4k+7}(\pd^{4k+4}_s \vk) + \mfq^{4k+5}(\pd^{4k+2}_s \vk) \\
 & \qquad  + \sum^{k}_{j=1} \{ \mfq^{4k+7-2j}(\pd^{4k+6-2j}_s \vk) + \mfq^{4k+5-2j}(\pd^{4k+4-2j}_s \vk) \} \\
 &= \mfq^{4(k+1)+3}(\pd^{4(k+1)}_s \vk) + \sum^{k+1}_{j=1} \mfq^{4(k+1)+3-2j}(\pd^{4(k+1)+2-2j}_s \vk). 
\end{align*}
This implies that \eqref{V-t-m} holds for any $m \in \N$. 
\end{proof}

We are in the position to prove the main result of this section. 
\begin{thm} \label{g-exe-g-flow}
Let $\lm \in \R$ be non-zero constant. 
Let $\gm_0 : I \to \R^2$ be a smooth open curve satisfying 
\begin{align*}
\gm_0(0)=(0,0), \,\, \gm_0(1)=(R,0), \,\, \vk_0(0)= \vk_0(1)= \va,  
\end{align*}
where $\va \in \R$ is a given constant with $\av{\va} < \av{\lm}$. 
Then there exists a unique family of smooth open planar curves $\gm(x,t)$ satisfying \eqref{g-curvature-flow} 
for any finite time $t>0$. 
\end{thm}
\begin{proof}
Suppose not, there exists a time $t_1>0$ such that the smooth solution $\gm(x,t)$ of \eqref{g-curvature-flow} 
remains up to $t=t_1$. 
Setting $\phi= \pd^m_t V^\lm$, Lemma \ref{l-e-method} implies that 
\begin{align} \label{Vlm-integ-eq}
&\dfrac{d}{dt} \dfrac{1}{4} \int_\gm (\pd^m_t V^\lm)^2 \, ds + \int_\gm (\pd^2_s \pd^m_t V^\lm)^2 \, ds \\
& \quad = \dfrac{1}{2}\int_\gm \pd^m_t V^\lm Y \, ds + \dfrac{1}{4} \int_\gm (\pd^m_t V^\lm)^2 \vk V^\lm \, ds. \notag
\end{align}
Regarding the integral of $(\pd^2_s \pd^m_t V^\lm)^2$, we have  
\begin{align*}
(\pd^2_s \pd^m_t V^\lm)^2 
& \geq (2^{2(m+1)} - \ve) (\pd^{4m+4}_s \vk)^2 \\
& \quad + \{ \mfq^{4m+5}(\pd^{4m+2}_s \vk) 
                     + \sum^{m}_{j=1} \mfq^{4m+5-2j}(\pd^{4m+4-2j}_s \vk) \}^2 \\
 &= c_m  (\pd^{4m+4}_s \vk)^2 +   \sum^{m}_{j=0} \mfq^{8m+10-2j}(\pd^{4m+2}_s \vk) \\
 & \qquad + \sum^{m}_{j, l=1} \mfq^{8m+10-2(j+l)}(\pd^{4m+4-2 \min{\{ j,l \}}}_s \vk). 
\end{align*}
Regarding the integral of $\pd^m_t V^\lm Y$, setting 
\begin{align*}
\pd^m_t V^\lm Y 
& =  \pd^m_t V^\lm \mfq^{4m+7}(\pd^{4m+4}_s \vk) 
       + \pd^m_t V^\lm \mfq^{4m+5}(\pd^{4m+4}_s \vk) \\
& \qquad + \pd^m_t V^\lm \sum^{m+1}_{j=2} \mfq^{4m+7-2j}(\pd^{4m+6-2j}_s \vk) 
 :=  I_1 + I_2 + I_3,    
\end{align*}
and integrating by part once the highest order term, we find 
\begin{align*}
\int_\gm I_1 \, ds  
 &= - \int_\gm \pd_s \pd^m_t V^\lm \{ \mfq^{4m+7}(\pd^{4m+3}_s \vk) + \mfq^{4m+6}(\pd^{4m+3}_s \vk) \} \, ds \\
 &= - \sum^{m}_{j=0} \int_\gm \{ \mfq^{8m+11-2j}(\pd^{4m+3}_s \vk) + \mfq^{8m+10-2j}(\pd^{4m+3}_s \vk) \} \, ds, 
\end{align*}
and 
\begin{align*}
\int_\gm I_2 \, ds  
 &= - \int_\gm \pd_s \pd^m_t V^\lm \{ \mfq^{4m+5}(\pd^{4m+3}_s \vk) + \mfq^{4m+4}(\pd^{4m+3}_s \vk) \} \, ds \\
 &= - \sum^{m}_{j=0} \int_\gm \{ \mfq^{8m+9-2j}(\pd^{4m+3}_s \vk) + \mfq^{8m+8-2j}(\pd^{4m+3}_s \vk) \} \, ds.  
\end{align*}
Hence we see that 
\begin{align*}
\int_\gm \pd^m_t V^\lm Y \, ds 
 = \int_\gm \sum^{m+1}_{j=0} \{ & \mfq^{8m+11-2j}(\pd^{4m+3}_s \vk) + \mfq^{8m+10-2j}(\pd^{4m+3}_s \vk)  \\
     & \quad + \mfq^{8m+8-2j}(\pd^{4m+2}_s \vk) + \mfq^{8m+6-2j}(\pd^{4m}_s \vk)\} \\
     & + \sum^{m}_{l=1, j=1} \mfq^{8m+8-2(j+l)}(\pd^{ 4m+2-2 \min{\{ j,l \}}}_s \vk) \, ds. 
\end{align*} 
Since it holds that 
\begin{align*}
 \int_\gm (\pd^m_t V^\lm)^2 \vk V^\lm \, ds 
& = \int_\gm \sum^{m+1}_{j=0} \{ \mfq^{8m+10-2j}(\pd^{4m+2}_s \vk) + \mfq^{8m+10-2j}(\pd^{4m}_s \vk) \} \\
    & \qquad + \sum^{m}_{l=1, j=1} \mfq^{8m+10-2(j+l)}(\pd^{ 4m+2-2 \min{\{ j,l \}}}_s \vk) \, ds, 
\end{align*}
the equality \eqref{Vlm-integ-eq} is reduced to 
\begin{align} \label{integ-diff}
& \dfrac{d}{dt} \dfrac{1}{4} \int_\gm (\pd^m_t V^\lm)^2 \, ds + c_m(\ve) \int_\gm (\pd^{4m+4}_s \vk)^2 \, ds \\
& \quad = \int_\gm \biggm[  \sum^{m+1}_{j=0} \Bigm\{ \mfq^{8m+11-2j}(\pd^{4m+3}_s \vk) + \mfq^{8m+10-2j}(\pd^{4m+3}_s \vk) \notag \\
& \qquad \qquad \qquad \qquad 
  + \mfq^{8m+10-2j}(\pd^{4m+2}_s \vk) + \mfq^{8m+10-2j}(\pd^{4m}_s \vk) \Bigm\} \notag \\
& \qquad \qquad \quad 
   + \sum^{m}_{l=1, j=1} \mfq^{8m+10-2(j+l)}(\pd^{ 4m+2-2 \min{\{ j,l \}}}_s \vk) \biggm] \, ds. \notag
\end{align}
We estimate the integral of $\mfq^{8m+11}(\pd^{4m+3}_s \vk)$ which is the highest order term in the right-hand side of  \eqref{integ-diff}. 
By Definition \ref{q-def}, this term can be written as  
\begin{align*}
\mfq^{8m+11}(\pd^{4m+3}_s \vk) 
 = \sum_{j} \prod^{N_j}_{i=1} \pd^{c_{j_{i}}}_s \vk 
\end{align*}
with all the $c_{j_{i}}$ less that or equal to $4m+3$, and 
\begin{align*}
\sum^{N_j}_{i=1} (c_{j_{i}} + 1)= 8m+11
\end{align*}
for every $j$. Hence we have 
\begin{align*}
\av{\mfq^{8m+11}(\pd^{4m+3}_s \vk)} 
 \leq \sum_{j} \prod^{N_j}_{i=1} \av{\pd^{c_{j_{i}}}_s \vk}.  
\end{align*}
Putting 
\begin{align*}
Q_j= \prod^{N_j}_{i=1} \av{\pd^{c_{j_{i}}}_s \vk}, 
\end{align*}
it holds that 
\begin{align*}
\int_\gm \av{\mfq^{8m+11}(\pd^{4m+3}_s \vk)} \, ds \leq \sum_{j} \int_\gm Q_j \, ds. 
\end{align*}
After collecting the derivatives of the same order in $Q_j$, we can write 
\begin{align*}
Q_j=  \prod^{4m+3}_{l=0} \av{\pd^l_s \vk}^{\va_{j_l}} \quad \text{with} \quad 
\sum^{4m+3}_{l=0} \va_{j_l} (l+1)= 8m+11. 
\end{align*}
Using H\"older's inequality we get  
\begin{align*}
 \int_\gm Q_j \, ds 
 \leq \prod^{4m+3}_{l=0} \left( \int_\gm \av{\pd^l_s \vk}^{\va_{j_l} \lm_l} \right)^{1/\lm_l} 
 = \prod^{4m+3}_{l=0} \Lp{\pd^l_s \vk}{\va_{j_l} \lm_l}^{\va_{j_l}}, 
\end{align*}
where the value $\lm_l$ are chosen as follows: $\lm_l=0$ if $\va_{j_l}=0$ (in this case 
the corresponding term is not present in the product) and $\lm_l= (8m+11)/\va_{j_l} (l+1)$ if $\va_{j_l} \neq 0$. 
Clearly $\va_{j_l} \lm_l = \tfrac{8m+11}{l+1} \geq \tfrac{8m+11}{4m+4} > 2$ and 
\begin{align*}
\sum^{4m+3}_{l=0, \lm_l \neq 0} \dfrac{1}{\lm_l} = \sum^{4m+3}_{l=0, \lm_l \neq 0} \dfrac{\va_{j_l} (l+1)}{8m+11}=1. 
\end{align*}
Let $k_l= \va_{j_l} \lm_l -2$. The fact $\va_{j_l} \lm_l > 2$ implies that $k_l>0$. 
Then we obtain 
\begin{align*}
\Lp{\pd^l_s \vk}{\va_{j_l} \lm_l} \leq c \Lp{\vk}{2}^{1-\sigma_{j_l}} \Lp{\vk}{4m+4,2}^{\sigma_{j_l}},
\end{align*}
where $\sigma_{j_l}= (l+ \tfrac{1}{2} - \tfrac{1}{\va_{j_l} \lm_l})/(4m+4)$ and $c= c(j,l,m)$. 
Since 
\begin{align*}
\Lp{\vk}{4m+4,2}^2 
 \leq C(m) \left( \Lp{\pd^{4m+4}_s \vk}{2}^2 + \Lp{\vk}{2}^2 \right), 
\end{align*}
we observe that 
\begin{align*}
\Lp{\pd^l_s \vk}{\va_{j_l} \lm_l} \leq C \Lp{\vk}{2}^{1-\sigma_{j_l}} \left( \Lp{\pd^{4m+4}_s \vk}{2}^2 + \Lp{\vk}{2}^2 \right)^{\sigma_{j_l}}. 
\end{align*}
Multiplying together all the estimates, we obtain 
\begin{align} \label{est-Qj}
 \int_\gm Q_j \, ds 
& \leq  C \prod^{4m+3}_{l=0} \Lp{\vk}{2}^{(1-\sigma_{j_l}) \va_{j_l}} 
                                         \left( \Lp{\pd^{4m+4}_s \vk}{2} + \Lp{\vk}{2} \right)^{\sigma_{j_l} \va_{j_l}} \\
&= C \Lp{\vk}{2}^{\sum^{4m+3}_{l=0}(1-\sigma_{j_l}) \va_{j_l}} 
                                         \left( \Lp{\pd^{4m+4}_s \vk}{2} + \Lp{\vk}{2} \right)^{\sum^{4m+3}_{l=0}\sigma_{j_l} \va_{j_l}}. \notag
\end{align}
Then the exponent in the last term of \eqref{est-Qj} is written as 
\begin{align*}
\sum^{4m+3}_{l=0}\sigma_{j_l} \va_{j_l} 
 = \sum^{4m+3}_{l=0} \dfrac{\va_{j_l} (l+ \tfrac{1}{2} - \tfrac{1}{\va_{j_l} \lm_l})}{4m+4} 
 = \dfrac{\sum^{4m+3}_{l=0}  \va_{j_l} (l+ \tfrac{1}{2}) - 1}{4m+4}, 
\end{align*}
and hence by using the rescaling condition we have 
\begin{align*}
\sum^{4m+3}_{l=0}\sigma_{j_l} \va_{j_l} 
 &= \dfrac{\sum^{4m+3}_{l=0}  \va_{j_l} (l+ 1) - \tfrac{1}{2}\sum^{4m+3}_{l=0} \va_{j_l} - 1}{4m+4} \\
 &= \dfrac{8m+11 - \tfrac{1}{2}\sum^{4m+3}_{l=0} \va_{j_l} - 1}{4m+4} 
   = \dfrac{16m+20 - \sum^{4m+3}_{l=0} \va_{j_l}}{2(4m+4)}.  
\end{align*}
Noting that 
\begin{align*}
\sum^{4m+3}_{l=0} \va_{j_l} 
 \geq  \sum^{4m+3}_{l=0} \va_{j_l} \dfrac{l+1}{4m+4} = \dfrac{8m+11}{4m+4}, 
\end{align*}
we see that 
\begin{align*}
\sum^{4m+3}_{l=0}\sigma_{j_l} \va_{j_l} 
 \leq \dfrac{16m+20 - \tfrac{8m+10}{4m+4}}{2(4m+4)} 
 = 2 - \dfrac{1}{(4m+4)^2} < 2. 
\end{align*}
Hence we can apply the Young inequality to the product in the last term of \eqref{est-Qj}, 
in order to get the exponent $2$ on the first quantity, that is, 
\begin{align*}
 \int_\gm Q_j \, ds 
&\leq \dfrac{\vd_j}{2}  \left( \Lp{\pd^{4m+4}_s \vk}{2} + \Lp{\vk}{2} \right)^2 + C_j \Lp{\vk}{2}^\vb \\
&\leq \vd_j \Lp{\pd^{4m+4}_s \vk}{2}^2 + \Lp{\vk}{2}^2 + + C_j \Lp{\vk}{2}^{\vb_j} 
\end{align*}
for arbitrarily small $\vd_j>0$ and some constant $C_j>0$ and exponent $\vb_j>0$. 
Hence we get 
\begin{align*}
& \dfrac{d}{dt} \dfrac{1}{4} \int_\gm (\pd^m_t V^\lm)^2 \, ds + \dfrac{1}{2} \int_\gm (\pd^2_s \pd^m_t V^\lm)^2 \, ds 
   + \dfrac{c_m(\ve)}{2} \int_\gm (\pd^{4m+4}_s \vk)^2 \, ds \\
& \quad \leq \sum^{m+1}_{j=0} \vd_j \Lp{\pd^{4m+4}_s \vk}{2}^2 +  C \sum^{m+1}_{j=0} \Lp{\vk}{2}^{\vb_j}. \notag
\end{align*}
Letting $\vd_j>0$ be sufficiently small, we obtain 
\begin{align} \label{diff-integ-eq}
\dfrac{d}{dt} \dfrac{1}{4} \int_\gm (\pd^m_t V^\lm)^2 \, ds \leq C \sum^{m+1}_{j=0} \Lp{\vk}{2}^{\vb_j}. 
\end{align}
Since Lemma \ref{l-E-bd-b} gives us 
\begin{align*}
\Lp{\vk}{2}^2 \leq C(\va, \lm) \mE_{\lm, \va}(\gm_0), 
\end{align*}
\eqref{diff-integ-eq} implies that 
\begin{align} \label{bd-V-t-m}
\Ln{\pd^m_t V^\lm(t)}{2}^2  \leq C_1 t +  \Ln{\pd^m_t V^\lm(0)}{2}^2 
\end{align}
for any time $t \in [0,t_1)$. 
Using \eqref{V-t-m} and the interpolation inequality, we reduce \eqref{bd-V-t-m} to 
\begin{align} \label{bd-k-s-4m+2}
\Ln{\pd^{4m+2}_s \vk}{2}^2 \leq C_1 t +  \Ln{\pd^m_t V^\lm}{2}^2(0) + C_2, 
\end{align}
where $C_2$ depends on $\mE_{\lm, \va}(\gm_0)$. 
Combining \eqref{p-E-lm-va} and \eqref{bd-k-s-4m+2} with the interpolation inequality, 
we observe that there exists a positive constant depending only on  $\mE_{\lm, \va}(\gm_0)$ such that 
\begin{align} \label{L2-bd-k-deri}
\Ln{\pd^l_s \vk}{2} \leq C_1 t +  \Ln{\pd^m_t V^\lm}{2}^2(0) + C_3 
\end{align}
for any $0 \leq l < 4m+2$. For each $l \in \N$, it is easy to obtain that 
\begin{align} \label{k-est-1}
\Ln{\pd^{l-1}_s \vk}{\infty} \leq C \Ln{\pd^l_s \vk}{1} + \mL(\gm)^{-1} \Ln{\pd^{l-1}_s \vk}{1}. 
\end{align}
Applying H\"older's inequality to \eqref{k-est-1}, we obtain 
\begin{align} \label{un-bd-k}
\Ln{\pd^{l-1}_s \vk}{\infty} \leq \mL(\gm)^{1/2} \Ln{\pd^l_s \vk}{2} + \mL(\gm)^{-1/2} \Ln{\pd^{l-1}_s \vk}{2}. 
\end{align}
Then it follows from \eqref{L2-bd-k-deri} and \eqref{un-bd-k} that there exists a constant $C=C(\gm_0, t_1, \va, \lm)$ such that 
\begin{align}
\Ln{\pd^{l-1}_s \vk(t)}{\infty} \leq C
\end{align}
for each $l \in \N$ and any $t \in [0, t_1)$. 
This contradicts that the solution of \eqref{g-curvature-flow} remains smooth to $t=t_1$. 
We thus complete the proof. 
\end{proof} 


\begin{center}
{\bf Acknowledgements}
\end{center}

The first author was partially supported by the Fondazione CaRiPaRo Project 
{\it Nonlinear Partial Differential Equations: models, analysis, and 
control-theoretic problems}. 
The second author was partially supported by Grant-in-Aid for Young Scientists (B) (No. 24740097). 
This work was done while the second author was visiting Centro di Ricerca Matematica Ennio De Giorgi 
and Department of Mathematics, University of Padova, 
whose hospitality he gratefully acknowledges. 




\begin{thebibliography}{99}

 \bibitem{a-g-s} L. Ambrosio, N. Gigli, and G. Savar\`e, 
{\it Gradient Flow}, Birkh\"auser, 2008.  

 \bibitem{BGN} J. W. Barrett, H. Garcke, and R. N\"urnberg,
 {\it Parametric approximation of isotropic and anisotropic elastic flow for closed and open curves}
 Numer. Math. {\bf 120} (2012), 489--542.
  
 \bibitem{b-m-n} G. Bellettini, C. Mantegazza, and M. Novaga, 
{\it Singular perturbations of mean curvature flow}, J. Differential Geom. 
{\bf 75} (2007), 403--431. 

 \bibitem{brezis} H. Brezis,
 {\it Op\'erateurs maximaux monotones et semi-groupes de contractions
dans les espaces de Hilbert}, North-Holland Publishing Co., Amsterdam, 1973.
 
  \bibitem{acqua-pozzi} A. Dall'Acqua and P. Pozzi, 
{\it A Willmore-Helfrich $L^2$-flow of curves with natural boundary conditions}, 
Preprint (2012). 
 
 \bibitem{dec-gru} K. Deckelnick and H.-C. Grunau, 
{\it Boundary value problem for the one-dimensional Willmore equation}, 
Calc. Var.  Partial Differential Equations {\bf 30} (2007), 293--314. 

 \bibitem{dziuk} G. Dziuk, E. Kuwert, and R. Sch\"atzle, 
{\it Evolution of elastic curves in $\R^n$: existence and computation}, 
SIAM J. Math. Anal. {\bf 33} (2002), 1228--1245. 

 \bibitem{koiso} N. Koiso, 
{\it On the motion of a curve towards elastica}, 
Acta de la Table Ronde de G\'eom\'etrie Diff\'erentielle, (1996), 403--436, 
S\'emin. Congr. 1, Soc. Math. France, Paris. 

 \bibitem{lin} C.-C. Lin, 
 {\it $L^2$-flow of elastic curves with clamped boundary conditions}, 
 J. Differential Equations {\bf 252} (2012), 6414--6428.   

 \bibitem{linner} A. Linn\'er, {\it Some properties of the curve straightening flow in the plane}, 
Trans. Amer. Math. Soc. {\bf 314} (1989), no. 2, 605--618. 

 \bibitem{linner-2} A. Linn\'er, {\it Explicit elastic curves}, Ann. Global Anal. Geom. 
{\bf 16} (1998), no. 2, 445--475. 

 \bibitem{novaga-okabe} M. Novaga and S. Okabe, 
 {\it Curve shortening-straightening flow for non-closed planar curves with infinite length}, Preprint (2012). 
 
 \bibitem{okabe-1} S. Okabe, {\it The motion of elastic planar closed curves under the area-preserving 
condition}, Indiana Univ. Math. J. {\bf 56} (2007), no. 4, 1871--1912. 

 \bibitem{okabe-2} S. Okabe, {\it The dynamics of elastic closed curves under uniform high pressure}, 
  Calc. Var. Partial Differential Equations {\bf 33} (2008), no. 4, 493--521. 

 \bibitem{polden} A. Polden, 
{\it Curves and Surfaces of Least Total Curvature And Fourth-Order Flows}, 
Dissertation University of Tuebingen (1996). 

 \bibitem{struwe} M. Struwe, {\it Variational Methods. Applications to Nonlinear Partial Differential Equations and 
Hamiltonian System. Fourth Edition}, Springer-Verlag, Berlin, 2008. 

 \bibitem{wen-1} Y. Wen, {\it $L^2$ flow of curve straightening in the plane}, 
 Duke Math. J. {\bf 70} (1993), no. 3, 683--698. 
 
 \bibitem{wen-2} Y. Wen, {\it Curve straightening flow deforms closed plane curves with nonzero 
rotation number to circles}, J. Differential Equations {\bf 120} (1995), no. 1, 89--107. 
 
 \bibitem{wheeler} G. Wheeler, {\it Global analysis of the generalised Helfrich flow of closed curves immersed in $\R^n$}, 
 Preprint (2013). 
 
\end{thebibliography}
\end{document}